\documentclass{amsart}
\usepackage{geometry}
\usepackage{mathrsfs}
\usepackage[T1]{fontenc}
\newcommand{\spf}{\mathrm{Spf}}
\usepackage{slashed}
\usepackage{verbatim}
\usepackage{amsmath,amsfonts,amsthm,amssymb,epsfig}
\usepackage[all]{xy}
\usepackage{enumitem}
\usepackage[usenames,dvipsnames]{xcolor}
\usepackage[bookmarks=true,colorlinks=true, linkcolor=Mulberry, citecolor=Periwinkle]{hyperref}
\usepackage{cleveref} 
\newcommand{\zpcycl}{\mathbb{Z}_p^{\mathrm{cycl}}}
\newcommand{\THH}{\mathrm{THH}}
\newcommand{\Dh}{\widehat{\mathcal{D}}}
\newcommand{\OmegaD}{\hat{\Omega}^{\slashed{D}}}

\newcommand{\wcartHT}{\mathrm{WCart}^{\mathrm{HT}}}

\newcommand{\qrsp}{\mathrm{qrsPerfd}}

\newcommand{\tN}{\widetilde{\mathcal{N}}}

\usepackage{graphicx}
\usepackage{float}

\newcommand{\et}{\mathrm{et}}

\newtheorem{theorem}{Theorem}[section]

\newtheorem{lemma}[theorem]{Lemma}
\newtheorem{proposition}[theorem]{Proposition}
\newtheorem{corollary}[theorem]{Corollary}

\theoremstyle{definition}

\newtheorem{construction}[theorem]{Construction}
\newcommand{\spec}{\mathrm{Spec}}
\newtheorem{definition}[theorem]{Definition}
\newtheorem{example}[theorem]{Example}

\newtheorem{notation}[theorem]{Notation}
\newtheorem{remark}[theorem]{Remark}

\newcommand\bcdot{\ensuremath{%
  \mathchoice%
   {\mskip\thinmuskip\lower0.2ex\hbox{\scalebox{1.5}{$\cdot$}}\mskip\thinmuskip}}%
   {\mskip\thinmuskip\lower0.2ex\hbox{\scalebox{1.5}{$\cdot$}}\mskip\thinmuskip}%
   {\lower0.3ex\hbox{\scalebox{1.2}{$\cdot$}}}%
   {\lower0.3ex\hbox{\scalebox{1.2}{$\cdot$}}}%
   }

\usepackage{relsize}
\usepackage[bbgreekl]{mathbbol}
\usepackage{amsfonts}
\DeclareSymbolFontAlphabet{\mathbb}{AMSb} 
\DeclareSymbolFontAlphabet{\mathbbl}{bbold}
\newcommand{\Prism}{{{\mathlarger{\mathbbl{\Delta}}}}}
\newcommand{\Prismnc}{{{\mathlarger{\mathbbl{\Delta}}}}}
\newcommand{\Prismbar}{{\overline{{\mathlarger{\mathbbl{\Delta}}}}}}

\newcommand{\Prismp}[2]{\frac{\Prism_{#1}\{#2\}}{p}}

\makeatletter
\@namedef{subjclassname@2020}{\textup{2020} Mathematics Subject Classification}
\makeatother
\begin{document}

\newcommand{\ksel}{K^{Sel}}
\newcommand{\qsyn}{\mathrm{qSyn}}
\newcommand{\mot}{\mathrm{mot}}
\title{Syntomic complexes and $p$-adic \'etale Tate twists}
\author{Bhargav Bhatt} 
\address{Department of Mathematics, University of Michigan, 530 Church Street,
Ann Arbor, MI 48109--1043}
\email{bhargav.bhatt@gmail.com}
\author{Akhil Mathew}
\address{Department of Mathematics, University of Chicago, 5734 S University
Ave., Chicago, IL, 60637}
\email{amathew@math.uchicago.edu}
\subjclass[2020]{14F30, 14F42}
\date{\today}

\maketitle

\begin{abstract}
The primary goal of this paper is to identify syntomic complexes with the
$p$-adic \'etale Tate twists of Geisser--Sato--Schneider on regular $p$-torsionfree schemes. Our methods apply naturally to a broader class of schemes that we call ``$F$-smooth''. The $F$-smoothness of regular schemes leads to new results on the absolute prismatic cohomology of regular schemes. 
\end{abstract}

\section{Introduction}

Let $X$ be a scheme. 
In \cite[Sec.~8]{BhattLurieAPC}, the first author and Lurie, following the
earlier work \cite{BMS2}, 
define and study certain \emph{syntomic complexes} $\mathbb{Z}_p(i)(X) = R \Gamma_{\mathrm{syn}}(X,
\mathbb{Z}_p(i))$ for $i \in \mathbb{Z}$, extending earlier constructions in the
literature \cite{FM87, Kato}. 
These syntomic complexes yield a generalization of the $p$-adic \'etale
cohomology (with Tate twisted coefficients) for $\mathbb{Z}[1/p]$-schemes to
arbitrary schemes, and exhibit quite different behavior in positive and mixed
characteristic, where they are obtained from prismatic cohomology. 
We refer to \cite[Sec.~1.1]{CN17} for a survey of
applications of syntomic cohomology. 

The purpose of this paper is to identify the syntomic complexes as \'etale
sheaves on $X$ in a class of examples. In doing so, we generalize a number of existing
results in the literature, including those of \cite{Kurihara, Kato, Tsuji,
CN17}, and recover the $p$-adic \'etale Tate twists of \cite{Schneider, Gei04,
Sato}. 

\subsection{What is syntomic cohomology?}

To formulate our results,  it is convenient to name the restriction of syntomic cohomology to the small \'etale site.

 \begin{notation}[The complexes $\mathbb{Z}/p^n(i)_X$] 
For any scheme $X$ and integer $i \in \mathbb{Z}$, write $\mathbb{Z}/p^n(i)_X
\in \mathcal{D}(X_{\mathrm{et}}, \mathbb{Z}/p^n)$ for the object of the derived
$\infty$-category of \'etale sheaves of $\mathbb{Z}/p^n$-modules on $X$ obtained
by restricting the syntomic complexes\footnote{In \cite[\S 8]{BhattLurieAPC},
the object $\mathbb{Z}_p(i)(X)$ is defined in the $p$-complete derived
$\infty$-category; by $\mathbb{Z}/p^n(i)(X)$ we mean the reduction mod $p^n$.
Note also that the construction $X \mapsto \mathbb{Z}/p^n(i)(X)$ on affine
schemes carries
filtered colimits 
of rings to filtered colimits \cite[Cor.~8.4.11]{BhattLurieAPC}.} $\mathbb{Z}/p^n(i)(-)$ of \cite[\S 8]{BhattLurieAPC} to the small \'etale site $X_{\mathrm{et}}$ of $X$. Thus, we have a defining identification $R\Gamma(X, \mathbb{Z}/p^n(i)_X) \simeq \mathbb{Z}/p^n(i)(X)$.
\end{notation} 

Let us describe this object in the key examples.

\begin{example}[Syntomic cohomology in characteristic $\neq p$]
\label{recovertatetwist}
For any $X$, the restriction of $\mathbb{Z}/p^n(i)_X$ to the locus $X[1/p]
\subset X$ is given by 
$\mu_{p^n}^{\otimes i} \simeq (\mathbb{Z}/p^n(i)_X)|_{X[1/p]}$.  In particular, if $p$ is invertible on $X$, then $\mathbb{Z}/p^n(i)_X$ is simply the usual \'etale Tate twist $\mu_{p^n}^{\otimes i}$. 
\end{example} 

\begin{example}[Syntomic complexes via logarithmic de Rham--Witt sheaves in characteristic $p$]
\label{syntomicequalchar}
When $X$ is a regular $\mathbb{F}_p$-scheme, we have isomorphisms
$\mathbb{Z}/p^n(i)_X \simeq W_n \Omega^i_{\mathrm{log}, X}[-i]$ for $W_n
\Omega^i_{\mathrm{log}, X}$ the
logarithmic Hodge--Witt sheaves considered in \cite{Milne76, Illusie79, Gros85},
cf.~\cite[Sec.~8]{BMS2}. 
\end{example} 

\begin{example}[Syntomic cohomology of $p$-adic formal schemes]
\label{ex:Syntomicpformal}
For any scheme $X$, 
the pullback of $\mathbb{Z}/p^n(i)_X$ to the \'etale site of the $p$-adic completion $\widehat{X}$ (or equivalently that of $X/p$) is constructed as  
a filtered Frobenius eigenspace of prismatic cohomology first studied in
\cite{BMS2}. 
That is, if $X = \spec(R)$ for $R$ a $p$-henselian animated ring, then 
one has an expression 
\begin{equation} \mathbb{Z}_p(i)(X) = 
\mathrm{eq} \left( \mathrm{can}, \phi_i: \mathcal{N}^{\geq i}
\Prism_X\left\{i\right\} \rightrightarrows \Prism_X\left\{i\right\} \right).
\end{equation}
Here $\Prism_X\left\{i\right\}$ denotes the Breuil--Kisin twisted (absolute)
prismatic cohomology of $X$, $\mathcal{N}^{\geq \ast}$ denotes the Nygaard
filtration, $\phi_i$ denotes the $i$th divided Frobenius, and $\mathrm{can}$
denotes the inclusion map. We refer to \cite[Sec.~7]{BhattLurieAPC} for a
detailed treatment of these objects.

Earlier versions of this construction (which agree with the above for $i \leq p-2$
or up to isogeny; cf.~\cite[Sec.~6]{AMNN} for comparisons) were introduced in \cite{FM87, Kato} using crystalline
cohomology and the Hodge filtration instead of prismatic cohomology and the
Nygaard filtration.  
\end{example} 

Examples~\ref{recovertatetwist} and \ref{ex:Syntomicpformal} essentially suffice to describe syntomic cohomology in general via a glueing procedure: if $R$ is a commutative ring with $p$-henselization $R^h_{(p)}$, one has a fibre square
\[ \xymatrix{ \mathbb{Z}/p^n(i)(\mathrm{Spec}(R)) \ar[r] \ar[d] &
R\Gamma_{\et}(\mathrm{Spec}(R[1/p]), \mu_{p^n}^{\otimes i}) \ar[d] \\
\mathbb{Z}/p^n(i)(\mathrm{Spec}(R^h_{(p)})) \ar[r] &
R\Gamma_{\et}(\mathrm{Spec}(R^h_{(p)}[1/p]), \mu_{p^n}^{\otimes i}), }\]
where the terms on the right are usual \'etale cohomology
(cf.~\Cref{recovertatetwist}), the term on the bottom left is computed via prismatic cohomology as in Example~\ref{ex:Syntomicpformal}, and the bottom horizontal map is obtained from the prismatic logarithm and the \'etale comparison theorem for prismatic cohomology in \cite[\S 8.3]{BhattLurieAPC}.  In fact, this approach was used as the definition of the top left vertex in \cite[\S 8.4]{BhattLurieAPC}. 

For any scheme $X$, the complex $\mathbb{Z}/p^n(0)_X$ identifies with the constant sheaf $\mathbb{Z}/p^n$ on $X_{\mathrm{et}}$. One can also make the complex explicit in weight $1$:

\begin{example}[{Cf.~\cite[Prop.~8.4.14]{BhattLurieAPC}}] 
\label{Zp1Gm}
For any scheme $X$, one has that $\mathbb{Z}/p^n(1)_X $ is the derived
pushforward of $\mu_{p^n}$ from the fppf site to the \'etale site (or
equivalently the fiber of $p^n: \mathbb{G}_m \to \mathbb{G}_m$ in the derived
category of \'etale sheaves). 
\end{example} 

Finally, for completeness, we recall that syntomic cohomology also has a close connection to
$p$-adic $K$-theory, yielding a simple construction of the former which appears
in \cite{Niziol12}. For this, we recall 
(\cite[Sec.~4]{BMS2} or \cite[App.~C]{BhattLurieAPC})
that a ring $R$ is \emph{$p$-quasisyntomic} if 
it has bounded $p$-power torsion and $L_{R/\mathbb{Z}} \otimes^{\mathbb{L}}_R
R/pR \in \mathcal{D}(R/pR)$ has $\mathrm{Tor}$-amplitude in $[-1, 0]$; for
instance, any lci noetherian ring has this property.

\begin{example}[The $\mathbb{Z}/p^n(i)_X$ via algebraic $K$-theory] 
Let $X$ be a  $p$-quasisyntomic scheme. 
In this case, one can give a direct construction of the $\mathbb{Z}/p^n(i)_X$
using algebraic $K$-theory for $i \geq 0$. 
Namely, $\mathbb{Z}/p^n(i)_X \in \mathcal{D}(X_{\mathrm{et}}, \mathbb{Z}/p^n)$
is the
derived pushforward of the sheafification of the presheaf $K_{2i}(-; \mathbb{Z}/p^n)$ from the syntomic site
of $X$ to the \'etale site of $X$. This is essentially a consequence of the work
\cite{BMS2} and  rigidity 
\cite{Gabber92, Su83, CMM}
and will be discussed in more detail separately. 
\end{example} 

The connection to algebraic $K$-theory does not play a direct role in this article; nonetheless, the connection to topological Hochschild homology provided by the $K$-theoretic approach inspired many of the arguments in this paper.

\subsection{Results}

Syntomic cohomology is essentially $p$-adic \'etale motivic
cohomology where the latter is defined,  cf.~\cite{Gei04, NiziolICM, EN19}.
For example, syntomic cohomology admits a robust theory of
Chern classes. However, the syntomic complexes are defined for arbitrary schemes through the theory of
prismatic cohomology, without any explicit use of algebraic cycles. 
We will identify syntomic cohomology for a class of $p$-torsionfree ``$F$-smooth'' schemes and obtain a formula related to the Beilinson--Lichtenbaum conjecture in
motivic cohomology. 
To begin, let us formulate the definition of $F$-smoothness.

\begin{definition}[$F$-smoothness, \Cref{filtSegaldef} below] 
We say that a $p$-quasisyntomic ring $R$ is {\em $F$-smooth}
if for each $i$, the prismatic divided Frobenius $\phi_i: \mathcal{N}^i \Prism_{{R}} \to
\Prismbar_{{R}} \left\{i\right\}$ has fiber in $\mathcal{D}(R)$ with $p$-complete
$\mathrm{Tor}$-amplitude in degrees $\geq i+2$, and if the Nygaard filtration
on the (twisted) prismatic cohomology 
$\Prism_R\left\{i\right\}$ is complete. This definition globalizes to schemes in a natural way. 
\end{definition} 

The terminology ``$F$-smooth'' is meant to evoke both the Frobenius (used in the
definition) as well as the hypothetical ``field with one element'': for
$p$-complete rings, we view $F$-smoothness roughly as an absolute version of the
smoothness condition in algebraic geometry.  Correspondingly, the class of
$F$-smooth rings contains smooth algebras over perfectoid rings
(\Cref{FSmoothPerfd}) and  for $p$-complete noetherian rings $F$-smoothness is equivalent to regularity (\Cref{regularringsSegal}). 
The verification that regular rings are $F$-smooth (and in particular the
Nygaard-completeness of the prismatic cohomology) has a further application:
under excellence assumptions, we verify a cohomological bound on the Hodge--Tate
stack of a regular local ring suggested in
\cite[Conj.~10.1]{BhattLurieprismatization}. 
In equal characteristic $p$, $F$-smoothness is
equivalent to the condition of Cartier smoothness identified in \cite{KM21,
KST21}. 
Over a perfectoid base, this condition has been studied independently in work of
V.~Bouis \cite{Bouis}; most of the following identification (\Cref{Zpiregularschemes})
of the $\mathbb{Z}/p^n(i)_X$ in this case has also been proved by Bouis,
cf.~\cite[Th.~4.14]{Bouis}. 

Let us now formulate the main comparison. 
By adjunction and \Cref{recovertatetwist}, for any scheme $X$, we have a 
natural map $\mathbb{Z}/p^n(i)_X \to Rj_* ( \mu_{p^n}^{\otimes i})$, for $j:
X[1/p] \subset X$ the open inclusion. 
For $i \geq 0$, results of \cite{AMNN} give that $\mathbb{Z}/p^n(i)_X \in
\mathcal{D}^{[0,
i]}(X_{\mathrm{et}}, \mathbb{Z}/p^n)$, whence we obtain a canonical comparison 
$\mathbb{Z}/p^n(i)_X \to \tau^{\leq i} Rj_* ( \mu_{p^n}^{\otimes i})$. 
In general, the Kummer map (obtained from \Cref{Zp1Gm} and the cup product) induces a map 
$(\mathcal{O}_X^{\times })^{\otimes i} \to \mathcal{H}^i ( \mathbb{Z}/p^n(i)_X)$
which one can show to be surjective; see also \cite{LM21} for more on the
target; this determines the image of $\mathcal{H}^i(
\mathbb{Z}/p^n(i)_X) \to R^i j_* ( \mu_{p^n}^{\otimes i})$ as the subsheaf
generated by $\mathcal{O}_{X}^{\times}$-symbols.\footnote{Note that by 
\cite{BK86, Hyodo88, SaitoSato}, 
the sheaf $R^i j_* ( \mu_{p^n}^{\otimes i})$ is generated by symbols from
$\mathcal{O}_{X[1/p]}^{\times}$ in a wide variety of settings.} 

\begin{theorem} 
\label{Zpiregularschemes}
Let $X$ be a $p$-torsionfree $F$-smooth scheme (e.g., a regular scheme flat over $\mathbb{Z}$). 
For $i \geq 0$, the comparison map $\mathbb{Z}/p^n(i)_X \to \tau^{\leq i} Rj_* (
\mu_{p^n}^{\otimes i})$ is an isomorphism on cohomology in degrees $< i$. 
On $\mathcal{H}^i$, the comparison map
is injective with image generated by the symbols, using the map of \'etale
sheaves
$(\mathcal{O}_X^{\times})^{\otimes i} \to \mathcal{H}^i( Rj_*
\mu_{p^n}^{\otimes i})  $. 
\end{theorem} 

In particular, $\mathbb{Z}/p^n(i)_X$ is obtained by modifying the truncated
$p$-adic nearby cycles $\tau^{\leq i} Rj_* ( \mu_{p^n}^{\otimes i})$ in the top
cohomological degree by taking the image of $(\mathcal{O}_X^{\times})^{\otimes
i}$: one has a fibre square
\[ \xymatrix{ \mathbb{Z}/p^n(i)_X \ar[r] \ar[d] & \tau^{\leq i} Rj_* ( \mu_{p^n}^{\otimes i}) \ar[d] \\
\mathrm{image}\left( (\mathcal{O}_X^{\times})^{\otimes i} \to R^i j_* (\mu_{p^n}^{\otimes i})\right)[-i] \ar[r] & R^i j_*( \mu_{p^n}^{\otimes i})[-i]  }\]
in $\mathcal{D}(X_{\mathrm{et}})$. On schemes which are smooth or regular with semistable reduction over a DVR, 
\Cref{Zpiregularschemes} identifies  $\mathbb{Z}/p^n(i)_X$  with the ``$p$-adic \'etale Tate twists'' considered in \cite{Sato},
and earlier in the smooth case in \cite{Gei04, Schneider}; cf.~\cite{Sato05}
for a survey.

Many special cases of \Cref{Zpiregularschemes} have previously appeared in the
literature. As above, the $\mathbb{Z}/p^n(i)_X$ always restrict to the usual Tate twists on
$X[1/p]$, so the main task is to identify $i^* \mathbb{Z}/p^n(i)_X$ for $i: X/p
\subset X$, or
equivalently the complexes defined in \cite{BMS2}. 
In low weights or up to isogeny (i.e., using the approach of \cite{FM87, Kato}), 
comparisons between syntomic
cohomology and $p$-adic vanishing cycles  
have been proved in a
variety of settings, including smooth and semistable schemes over a DVR or its
absolute integral closure, in \cite{Kurihara, Kato, Tsuji, CN17}. 
\Cref{Zpiregularschemes} integrally in all weights for smooth
$\mathcal{O}_C$-algebras, 
 for $C$ an algebraically closed
complete nonarchimedean field of mixed characteristic $(0, p)$, is
proved in \cite[Sec.~10]{BMS2} (see also \cite{CDN21} for a semistable analog). 

\Cref{Zpiregularschemes} is also closely related (via \cite{BMS2}) to the calculations of
topological cyclic homology for smooth algebras over the ring of integers in a
$p$-adic field, cf.~\cite{HM03, HM04, GH06}, and the recent revisiting in
\cite{LW21}. We do
not calculate the topological cyclic homology but rather its associated graded
terms, and the methods are at least superficially different; it would be interesting to make a
direct connection.\footnote{The Segal conjecture at the level of topological
Hochschild homology, which is closely related to the condition of
$F$-smoothness,  is often used in these calculations.} 

Our proof of \Cref{Zpiregularschemes} is based on some calculations in prismatic
cohomology. In particular, it is based on the \'etale comparison theorem
(cf.~\cite[Th.~9.1]{Prisms}, \cite[Th.~8.5.1]{BhattLurieAPC}, and \Cref{etcompthm}
below), which
states that for any scheme $X$, one can always recover the Tate twists
$\mu_{p}^{\otimes i}$ on the generic fiber by inverting a suitable class $v_1
\in H^0( \mathbb{F}_p(p-1)(\mathbb{Z}))$ in the syntomic cohomology of $X$. One can identify the image of the
class $v_1$ in the prismatic cohomology of $\mathbb{Z}_p$, after which the
result follows from a linear algebraic argument.

\subsection*{Conventions}
Throughout, we use the theory of (absolute) prismatic cohomology
as developed in \cite{BhattLurieAPC, BMS2, Drinfeld, Prisms}. 

We will simply write $\hat{R}$ for the $p$-adic completion if there is no potential for confusion.
If $R$ is $p$-complete, we write $R\left \langle t\right\rangle$ for the
$p$-completed polynomial ring and $R\left \langle t^{1/p^\infty}\right\rangle$
for the $p$-completion  of $R[t^{1/p^\infty}]$. 

For an animated ring $R$, we let $\mathcal{D}(R)$ denote the $\infty$-category
of $R$-modules (i.e., if $R$ is an ordinary ring, $\mathcal{D}(R)$ is the
derived $\infty$-category of $R$). 

Given an object $M \in \mathcal{D}(R)$ and an element $x \in R$, we will write
$M/x$ or $\frac{M}{x}$ for the mapping cone of $x: M \to M$. In particular, even
when $M$ is a discrete $R$-module, the object $M/x$ need not live in degree $0$.

\subsection*{Acknowledgments}
 Dustin Clausen was involved in earlier stages of this project, and we thank him heartily for many helpful discussions and insights. We also thank Benjamin Antieau, K\k{e}stutis \v{C}esnavi\v{c}ius, Jeremy Hahn, Lars Hesselholt, Jacob Lurie, Wies\l awa Nizio\l, Matthew Morrow,  Peter Scholze, and Dylan Wilson for helpful discussions. In particular, we learned the idea that there should be a mixed characteristic analog of Cartier smoothness from Morrow, and some of the arguments used by Hahn--Wilson in \cite{HW21} inspired some of ours used here. 
 Finally, we thank the referee for many helpful comments and corrections on an
 earlier version of this paper. 

This work was done while the first author was partially supported by the NSF
(\#1801689, \#1952399, \#1840234), the Packard Foundation, and the Simons
Foundation (\#622511), and while the second author was a Clay Research Fellow
and partially supported by the NSF (\#2152235). 

\section{Some calculations in prismatic cohomology}
\label{sec:calculations}

In this section, we recall some basic calculations 
in absolute prismatic cohomology. 
Our goal is to name some elements $v_1, \widetilde{\theta}, \theta$ in the
prismatic cohomology of $\mathbb{Z}_p$, which will play a basic role in the
sequel. 

\subsection{Prismatic sheaves}

Let us first 
recall the construction of the prismatic sheaves, after \cite{BhattLurieAPC,
Prisms}; their Nygaard completion was first constructed in \cite{BMS2}.

Following \cite[Sec.~4]{BMS2}, we use the \emph{quasisyntomic site}
$\qsyn_{\mathbb{Z}_p}$. An object of $\qsyn_{\mathbb{Z}_p}$ is a
$p$-complete, $p$-torsionfree
ring $A$ such that $L_{A/\mathbb{Z}_p} \otimes_A^{\mathbb{L}} (A/p) \in
\mathcal{D}(A/p)$ has $\mathrm{Tor}$-amplitude in $[-1, 0]$. 
There is a basis 
$\qrsp_{\mathbb{Z}_p} \subset \qsyn_{\mathbb{Z}_p}$ of $p$-torsionfree
\emph{quasiregular semiperfectoid rings}, i.e., those objects in
$\qsyn_{\mathbb{Z}_p}$ which admit a surjection
from a perfectoid ring. 

\begin{construction}[Prismatic sheaves] 
Let $R \in \qrsp_{\mathbb{Z}_p}$ be a $p$-torsionfree quasiregular semiperfectoid ring. Then we have naturally
associated to $R$ the following: 
\begin{enumerate}
\item A prism $( \Prism_R, \phi, I)$ together with a map $R \to \Prism_R/I$ (which is
in fact the initial prism with this structure). We write $\Prismbar_R =
\Prism_R/I$ and call it the Hodge--Tate cohomology. 
\item An invertible $\Prism_R$-module $\Prism_R\left\{1\right\}$ 
with a natural $\phi$-linear map $\phi_1 : \Prism_R\left\{1\right\}
\to I^{-1} \Prism_R\left\{1\right\}$ whose $\phi$-linearization is an
isomorphism; the reduction $\Prismbar_R\left\{1\right\}$ is identified with
$I/I^2$. 
We let $\Prism_R\left\{n\right\} = \Prism_R\left\{1\right\}^{\otimes n}$ and
obtain
$\phi_n: \Prism_R\left\{n\right\} \to I^{-n} \Prism_R\left\{n\right\}$. 
\item A descending, multiplicative Nygaard filtration $\left\{\mathcal{N}^{\geq
i } \Prism_R\right\}$ on the ring $\Prism_R $ given by $\mathcal{N}^{\geq i} \Prism_R = \phi^{-1}(I^i \Prism_R)$; we write
$\mathcal{N}^i \Prism_R = \mathrm{gr}^i (\mathcal{N}^{\geq \ast} \Prism_R)$. 
\item 
A map of graded rings $\bigoplus_{i \geq 0} \mathcal{N}^i \Prism_R \to
\bigoplus_{i \in \mathbb{Z}} (I/I^2)^{\otimes i} = \bigoplus_{i \in \mathbb{Z}}
\Prismbar_R\left\{i\right\}$, obtained by passing to associated graded terms of
the map of filtered rings 
$\phi: \left\{\mathcal{N}^{\geq \ast} \Prism_R\right\} \to \left\{I^{\ast}
\Prism_R\right\}$. 
\item The prismatic logarithm 
$\log_{\Prism}: T_p(R^{\times}) \to \Prism_R\left\{1\right\}$, whose image consists precisely
of those elements $y \in \Prism_R\left\{1\right\}$ such that $\phi_1(y) =y$. 
\end{enumerate}

All of the above define sheaves of $p$-torsionfree, $p$-complete abelian groups
with trivial higher cohomology on 
$\qrsp_{\mathbb{Z}_p}$; by descent, one obtains
$\Dh(\mathbb{Z}_p)$-valued sheaves on $\qsyn_{\mathbb{Z}_p}$ with the
same notation. 
Moreover, we will also need to consider the prismatic complexes 
for arbitrary animated rings; these can be defined starting from the above using
animation (compare
\cite[Sec.~4.5]{BhattLurieAPC}). 
\end{construction}

\begin{construction}[Syntomic sheaves] 
One has also, for each $i \geq 0$,  the $\mathcal{D}(\mathbb{Z}_p)^{\geq
0}$-valued sheaf of abelian groups $\mathbb{Z}_p(i)(-)$ on
$\qrsp_{\mathbb{Z}_p}$ which carries $R$ to the fiber of 
$\mathrm{can}- \phi_i: \mathcal{N}^{\geq i} \Prism_R\left\{i\right\} \to
\Prism_R\left\{i\right\}$ for $\mathrm{can}$ the inclusion map, as
originally introduced in \cite{BMS2}. By \cite[Th.~14.1]{Prisms}, there is a
basis for $\qrsp_{\mathbb{Z}_p}$ on which the $\mathbb{Z}_p(i)(-)$ are discrete. 

By animation, one extends the $\mathbb{Z}_p(i)(-)$ by animation to all $p$-complete animated rings. 
In \cite[Sec.~8]{BhattLurieAPC}, the syntomic sheaves $\mathbb{Z}_p(i)(-)$ are extended to all
animated rings, and by Zariski descent to all schemes, by gluing the above
construction on the $p$-completion and the usual Tate twists on the generic
fiber. On $p$-quasisyntomic rings, the $\mathbb{Z}_p(i)(-)$ are concentrated in
nonnegative degrees. 
\end{construction}

\begin{example}[The case of $\zpcycl$] 
\label{caseofzpcycl}
In the particular case where $R =
\zpcycl \stackrel{\mathrm{def}}{=}\widehat{\mathbb{Z}_p[\zeta_{p^\infty}]}$,
then we have an identification $\Prism_{R} =
\widehat{\mathbb{Z}_p[q^{1/p^\infty}]}_{(p, q-1)}$, $I = [p]_q :=
\frac{q^p-1}{q-1}$. In this case, the choice of $p$-power roots $(1, \zeta_p,
\zeta_{p^2}, \dots )$ determines an element $\epsilon \in T_p(R^{\times})$ such
that $\log_{\Prism}(\epsilon) \in \Prism_R\left\{1\right\}$ is divisible by
$(q-1)$, and such that $\frac{\log_{\Prism}(\epsilon)}{q-1}$ is a generator for
the module $\Prism_R\left\{1\right\}$, cf.~\cite[Sec.~2.6]{BhattLurieAPC}. 
\end{example}

\begin{construction}[The Hodge--Tate cohomology of $\mathbb{Z}_p$] 
\label{HTofZp}
Let us recall the calculation of the Hodge--Tate cohomology of $\mathbb{Z}_p$.
In fact, we have an isomorphism of bigraded $\mathbb{F}_p$-algebras,
\[ H^*\left(  \frac{\Prismbar_{\mathbb{Z}_p}}{p} \left\{\ast\right\}\right)
\simeq E( \alpha) \otimes
P(\theta^{\pm 1})  \]
where $|\alpha| = (1, p)$ and $\theta = (0, p)$ (we write the cohomological
grading first and the internal grading next). 
In fact, this follows from the treatment in \cite[Sec.~3]{BhattLurieAPC}. 
The Hodge--Tate cohomology of $\mathbb{Z}_p$ is given by the coherent cohomology
of the sheaves $\mathcal{O}_{\wcartHT}\left\{i\right\}$
on the stack $\wcartHT \simeq B \mathbb{G}_m^{\sharp}$.  
As in \emph{loc.~cit.}, $p$-torsion sheaves on $B \mathbb{G}_m^{\sharp}$ are
simply $\mathbb{F}_p$-vector spaces $V$ equipped with an endomorphism $\Theta: V
\to V$ such that the generalized eigenvalues of $\Theta $ live in $\mathbb{F}_p
\subset \overline{\mathbb{F}_p}$, and $\mathcal{O}_{\wcartHT}\left\{i\right\}$
corresponds to the endomorphism $i: \mathbb{F}_p \to \mathbb{F}_p$. With this
identification in mind, the calculation follows. 

Using \cite[Prop.~5.7.9]{BhattLurieAPC}, we also find 
\[ H^*\left( \bigoplus_{i \geq 0} \frac{\mathcal{N}^i \Prism_{\mathbb{Z}_p}}{p}
\right) \simeq
E(\alpha) \otimes P(\theta)  \]
such that the natural map 
$\bigoplus_{i \geq 0} \frac{\mathcal{N}^i \Prism_{\mathbb{Z}_p}}{p} \to \bigoplus_{i
\in \mathbb{Z} } \frac{\Prismbar_{\mathbb{Z}_p}}{p}\left\{i\right\}$
on cohomology carries $\alpha \mapsto \alpha, \theta \mapsto
\theta$.\footnote{Under the motivic filtrations of \cite{BMS2}, this calculation is also closely
related to B\"okstedt's calculation of $\THH_*(\mathbb{Z})$.}
\end{construction}

\begin{example}
Let $R$ be a $p$-torsionfree perfectoid ring. 
We have  $R \xrightarrow{\sim} \Prismbar_R$, so one forms the Breuil--Kisin
twists $R\left\{i\right\}$. 
The map 
$\bigoplus_{i \geq 0} \mathcal{N}^i \Prism_R \to \bigoplus_{i \in \mathbb{Z}}
\Prismbar_R\left\{i\right\}$ is identified with the inclusion map 
$\bigoplus_{i \geq 0} R\left\{i\right\} \to \bigoplus_{i \in \mathbb{Z}}
R\left\{i\right\}$. 
Under these identifications, 
$\theta$ maps to a generator of $\mathcal{N}^p \frac{\Prism_R}{p}$; in fact,
this is evident because $\theta$ is a unit in the Hodge--Tate
cohomology. \end{example}

\begin{proposition} 
\label{invertingtheta}
Let $A$ be any animated ring. Then the map of graded $E_\infty$-rings over
$\mathbb{F}_p$,
\[ \bigoplus_{i \geq 0} \mathcal{N}^i \frac{\Prism_A}{p} \to \bigoplus_{i \in
\mathbb{Z}} \frac{\Prismbar_A\left\{i\right\}}{p}  \]
exhibits the target as the localization of the source at the element $\theta$. 
\end{proposition} 
\begin{proof} 
By quasisyntomic descent and left Kan extension, it suffices to treat the case
where $A$ is a smooth algebra over a $p$-torsionfree perfectoid ring, so that
one is in the setting of relative prismatic cohomology \cite{Prisms}. 
In this case, one can trivialize the Breuil--Kisin twists, and one knows that the map 
$\phi_i : \mathcal{N}^i \Prism_A \to \Prismbar_A\left\{i\right\}$ is the $i$th
stage of the conjugate filtration on the Hodge--Tate cohomology $\Prismbar_A
\simeq \Prismbar_A\left\{i\right\}$, cf.~\cite[Th.~12.2]{Prisms}. Since the conjugate filtration is
exhaustive and since $\theta$ maps to a unit in the target, the result easily
follows from the Hodge--Tate comparison \cite[Th.~4.11]{Prisms}. 
\end{proof}

\subsection{The elements $v_1$, $\widetilde{\theta}$}
In this subsection we construct two further elements in the prismatic cohomology
of $\mathbb{Z}$. 
\begin{construction}[The class $v_1$] 
\label{constructionofv1}
We define a class $v_1 \in H^0( \mathbb{F}_p(p-1)(\mathbb{Z}))$ as follows. 

Let $R$ be the ring $\mathbb{Z}[\zeta_{p^\infty}]$. 
Then by flat descent \cite[Prop.~8.4.6]{BhattLurieAPC}, 
$H^0( \mathbb{F}_p(p-1)(\mathbb{Z}))$ is the equalizer of the two maps
\begin{equation} \label{twomapsR} H^0( \mathbb{F}_p(p-1)(R)) \rightrightarrows H^0( \mathbb{F}_p(p-1)(R
\otimes R)).\end{equation}
The element $(1,
\zeta_p, \zeta_{p^2}, \dots ) \in T_p( R^{\times})$ 
determines a class $ \epsilon \in H^0( \mathbb{Z}_p(1)(R))$ via
the identification of
\cite[Prop.~8.4.14]{BhattLurieAPC}. 
We claim that the image of 
$\epsilon^{p-1} \in H^0( \mathbb{F}_p(p-1)(R))$ belongs to the equalizer of the
two maps \eqref{twomapsR}. 

To see this, it suffices to map $R \otimes R$ to both its $p$-adic
completion and to $R \otimes R[1/p]$. The images of
$\epsilon^{p-1}$ in the latter are identical, as one sees using the trivialization of
the sheaf
$\mu_p^{ \otimes p-1}$
on $\mathbb{Z}[1/p]$-algebras. 
Thus, it suffices to calculate in 
$\mathbb{F}_p(p-1) ( \widehat{R \otimes R})$. 
Equivalently, we may do this calculation in $\Prism_{R \otimes R}/p\left\{p-1\right\}$. 
By construction, the two images of $\epsilon$ yields classes
$\epsilon_1, \epsilon_2 \in T_p\left((\widehat{R \otimes
R})^{\times}\right)$. 
The images under the prismatic logarithm mod $p$ yield elements 
$$\log_{\Prism}(\epsilon_1), \log_{\Prism}(\epsilon_2) \in  \Prism_{R \otimes
R}\left\{1\right\}/p.$$
As in \Cref{caseofzpcycl}, $\Prism_R$ is
canonically identified with $\widehat{\mathbb{Z}_p[q^{1/p^\infty}]}_{(p, q-1)}$. 
Let $q_1, q_2 \in \Prism_{R \otimes R}$ denote the images of $q$
under the two maps $\Prism_{R} \rightrightarrows \Prism_{R \otimes
R}$. 

Since the maps are $(p,I)$-completely flat, the elements $(q_1-1), (q_2 - 1) \in \Prism_{R \otimes
R}/p$ are nonzerodivisors, by the conjugate filtration and the Hodge--Tate
comparison \cite[Th.~4.11]{Prisms}. 
To see that $\log_{\Prism}( \epsilon_1)^{p-1} =\log_{\Prism}( \epsilon_2)^{p-1}
 \in \Prism_{R \otimes R}\left\{p-1\right\}/p$, we may thus invert $(q_1 -
 1)(q_2-1)$, after which both 
 $\log_{\Prism}(\epsilon_1)$ and $\log_\Prism(\epsilon_2)$ become generators of 
 the invertible $\Prism_{R \otimes R}/p[\frac{1}{(q_1-1)(q_2-1)}]$-module 
 $\Prism_{R \otimes R}\left\{1\right\}/p[\frac{1}{(q_1-1)(q_2-1)}]$. But then there exists a unit $x \in \Prism_{R \otimes R}/p [ \frac{1}{(q_1-1)(q_2-1)}]$ with $x \log_{\Prism}(\epsilon_1) = \log_{\Prism}(\epsilon_2)$. 
Since $\log_{\Prism}(\epsilon_i), i = 1, 2$ are fixed points of the divided
Frobenius $\phi_1$, we find that $\phi(x) = x$, or $x^{p} = x$. Since $x$ is a unit, this gives  $x^{p-1} =1$, so $\log_\Prism(\epsilon_1)^{p-1} = \log_{\Prism}(\epsilon_2)^{p-1}$ in $\Prism_{R \otimes
R}/p\left\{p-1\right\}[\frac{1}{(q_1-1)(q_2-1)}]$, as desired. 
\end{construction} 

The class $v_1 \in H^0( \mathbb{F}_p(p-1)(\mathbb{Z}_p))$ also appears (in a different language) in
\cite[Prop.~8.11.2]{Drinfeld}.

Although it will not play a role in the sequel, let us remark on the connection
to the element $v_1$ in stable homotopy theory. 
Suppose $p>2$ for simplicity. 
The topological class $v_1^{\mathrm{top}} \in \pi_{2p-2}(\mathbb{S}/p)$ in
the stable stems gives a nonzero class in $\pi_{2p-2}
\mathrm{TC}(\mathbb{Z}_p; \mathbb{F}_p)$; under the motivic spectral sequence of 
\cite{BMS2}, this is detected (up to nonzero scalar) by the class denoted $v_1$
above. 
In fact, we can check this after passage from $\mathbb{Z}_p$ to $
\mathcal{O}_{\mathbb{C}_p}$; 
then, the description $ku/p=\mathrm{TC}(\mathcal{O}_{\mathbb{C}_p}; \mathbb{F}_p)$ 
(cf.~\cite{HN19} for an account)
easily implies the claim.

\begin{construction}[The element $\widetilde{\theta}$] 
The element $v_1 \in H^0( \mathbb{F}_p(p-1)(\mathbb{Z}))$ maps to 
$H^0\left( \mathcal{N}^{\geq p-1} \Prismp{\mathbb{Z}_p}{p-1} \right)$. In fact, 
since $\mathcal{N}^{p-1} \Prism_{\mathbb{Z}_p} = 0$ (\Cref{HTofZp}), we obtain a unique
lift
to an element 
$\widetilde{\theta} \in H^0\left( \mathcal{N}^{\geq p}
\Prismp{\mathbb{Z}_p}{p-1}\right)$. 
\end{construction} 

\begin{proposition} 
\label{widetildethetaimage}
The image of $\widetilde{\theta}$ in $H^0 ( \mathcal{N}^p
\Prism_{\mathbb{Z}_p}/p)$ is a generator (which, up to normalization, we can
take to be $\theta$). 
\end{proposition} 
\begin{proof} 
It suffices to show that the image of $\widetilde{\theta}$ is nonzero in 
$H^0 ( \mathcal{N}^p
\Prism_{\mathbb{Z}_p}/p) $. 
We may do this calculation in $\mathbb{Z}_p^{\mathrm{cycl}}$. 
Let $\epsilon \in T_p( (\zpcycl)^{\times})$ be the 
canonical element $(1, \zeta_p, \zeta_{p^2}, \dots )$. 
We  have $v_1 = \mathrm{log}_{\Prism}(\epsilon)^{p-1}$, which is $(q-1)^{p-1}
\equiv (q^{1/p} -
1)^{p(p-1)} \ (\mathrm{mod} \ p)$ times a
generator of $\Prism_{\mathbb{Z}_p^{\mathrm{cycl}}}\left\{p-1\right\}/p$. 
Noting that the Nygaard filtration is the filtration by powers of 
$[p]_{q^{1/p}} \equiv (q^{1/p} - 1)^{p-1} \ (\mathrm{mod} \ p ) $, we find that $v_1$ maps to a nonzero
element of $\mathcal{N}^p \Prismp{\zpcycl}{p-1}$, as desired. 
\end{proof}

\begin{remark}[A direct prismatic construction] 
Let us now describe another construction of the  image of $v_1 $ in $ H^0( \mathbb{F}_p(p-1)(\mathbb{Z}_p))$ that does
not rely on the explicit use of the ring $\mathbb{Z}[\zeta_{p^\infty}]$ or the
prismatic logarithm. 
Given any  $p$-torsionfree prism $(A, I, \phi)$ such that $A/I$ is also
$p$-torsionfree, 
we have as in \cite[Sec.~2.2]{BhattLurieAPC} a natural invertible module
$A\left\{1\right\}$ together with a $\phi$-linear
map
$\phi_1: A\left\{1\right\} \to I^{-1} A\left\{1\right\}$ which 
becomes an isomorphism upon $\phi$-linearization. 
We also have the tensor powers $A\left\{i\right\}$ and the maps $\phi_i:
A\left\{i\right\} \to I^{-i} A\left\{i\right\}$. 
Specifying an element of $H^0(
\mathbb{F}_p(p-1)(\mathbb{Z}_p))$ is equivalent to specifying, for each such
prism $(A, I)$, an element of $A/p\left\{p-1\right\}$ which is fixed under
$\phi_{p-1}$. 

Let us construct an element in $I A/p\left\{p-1\right\}$ which is a fixed point
for $\phi_{p-1}: A/p\left\{p-1\right\} \to I^{-(p-1)} A/p$, as follows. 
Choose a generator $y \in A/p\left\{1\right\}$. 
By the above, $\phi_1(y)/y$ is a generator for the invertible $A/p$-module
$I^{-1}/p$, so ``$y/\phi_1(y)$'' is a generator for the ideal $I/p \subset A/p$. 
Now consider the element
$\frac{y}{\phi_1(y)}  y^{p-1} \in I A/p\left\{p-1\right\}$. 
Unwinding the definitions, it follows that 
$\phi_{p-1}$ carries this element to 
$\frac{y^{\otimes p}}{\phi_1(y)^{\otimes p}} \phi_{p-1}(y^{\otimes p-1})  =
\frac{y}{\phi_1(y)} \otimes y^{\otimes p-1}$, i.e., we have a fixed point for
$\phi_{p-1}$. It is easy to check that this does not depend on the choice of
generator $y$ and that it produces a  fixed point for $\phi_{p-1}$ (modulo $p$)
as desired. 
One can check that this construction reproduces the image of $v_1$ in $H^0(
\mathbb{F}_p(p-1)(\mathbb{Z}_p))$ at least up to scalars, by calculating explicitly
for the prism corresponding to the perfectoid ring $\zpcycl$. 
\end{remark} 

\section{The Nygaard filtration on Hodge--Tate cohomology}

In this section, we define the Nygaard filtration on Hodge--Tate cohomology
and study some of its basic properties. 
\subsection{Definitions}
\begin{construction}
\label{NygaardHT}
Let $R \in \qrsp_{\mathbb{Z}_p}$. 
Consider the prism $( \Prism_R, I)$ and the Nygaard filtration 
$\mathcal{N}^{\geq \ast} \Prism_R$. The  image of the Nygaard filtration yields a
filtered ring $\mathcal{N}^{\geq \ast} \Prismbar_R$. 
The ideal $I \subset \Prism_R$ maps via the canonical augmentation $\Prism_R \to
R$ to the ideal $p$ (e.g., by calculating explicitly for $R = \zpcycl$). 
Therefore, we have a canonical isomorphism  of graded rings
\begin{equation} 
\label{grNygaardHT}
 \mathrm{gr}^{\ast} \Prismbar_R \simeq \bigoplus_{i \geq 0} \mathcal{N}^i
\frac{\Prism_R}{p}.  
\end{equation} 
Note here the composite of $R \to \Prismbar_R \to \mathrm{gr}^0 \Prismbar_R \simeq R/p$ is
the Frobenius. 
In particular, if we consider the filtration \eqref{grNygaardHT} as one of
$R$-modules, then 
$\mathrm{gr}^i \Prismbar_R \simeq \mathcal{N}^i \Prism_R/p^{(-1)}$, with the
superscript denoting restriction along Frobenius. We highlight the special case
of an isomorphism of $R$-algebras,
\begin{equation}  \label{gr0prismbar} \mathrm{gr}^0 \Prismbar_R \simeq R/p^{(-1)} , \end{equation}
for $R \in \qrsp_{\mathbb{Z}_p}$, and then by descent and left Kan extension for
all animated rings $R$. 
We can also do the same with the Breuil--Kisin twists
$\Prismbar_R\left\{i\right\}$, which yield invertible
$\mathcal{N}^{\geq \ast}\Prismbar_R$-modules
$\mathcal{N}^{\geq \ast} \Prismbar_R\left\{i\right\}$, 
with associated gradeds the same as above. 

By descent and Kan extension, we construct for any animated ring $A$ the
commutative algebra object 
$\mathcal{N}^{\geq \ast}\Prismbar_A$ of
the filtered derived $\infty$-category. 
\end{construction}

In the remainder of the subsection, we detect the element $p$ in the Nygaard
filtration on Hodge--Tate cohomology, and obtain a twisted form of the
filtration for Hodge--Tate cohomology modulo $p$ which will sometimes be easier
to work with.

\begin{example}[Detection of the element $p$] 
\label{identificationofp}
We show that the element $p \in H^0( \Prismbar_{\mathbb{Z}_p})$ is detected in filtration $p$
of the Nygaard filtration on 
$\Prismbar_{\mathbb{Z}_p}$, by the class $\theta \in H^0 ( \mathcal{N}^p
\frac{\Prism_{\mathbb{Z}_p}}{p})$ (up to units). 

To see this, 
we may replace $\mathbb{Z}_p$ by the perfectoid ring $R =
\widehat{\mathbb{Z}_p[p^{1/p^\infty}]}$, and it suffices to show that $p \in
\mathcal{N}^{\geq p} \Prismbar_R \setminus \mathcal{N}^{\geq p+1} \Prismbar_R$. 
Since $R$ is perfectoid, $\Prism_R = W(R^{\flat})$. 
Let $p^{\flat} \in R^{\flat}$ be given by the system of elements $(p, p^{1/p},
p^{1/p^2}, \dots ) $ in $R$. 
The prismatic ideal $I \subset \Prism_R  = W(R^{\flat})$ is  $I = (p -
[p^{\flat}])$, and the 
map 
$R \to \Prism_R/I$ is an isomorphism whose inverse given by the Fontaine map
$W(R^{\flat}) \to R$ (whose kernel is $I$). 
Now
$\mathcal{N}^{\geq i} \Prism_R = \phi^{-1}(I)^i = (p - [p^{\flat, 1/p}])^i$. 
The image of this ideal in $\Prismbar_R$ is $p^{i/p}$, since $[p^{\flat, 1/p}]$
maps to $p^{1/p}$. The claim now follows. 
\end{example}

\begin{construction}[The twisted Nygaard filtration on $\frac{\Prismbar_R}{p}$]
\label{TwistedNygaard}
Let $R$ be any animated ring. 
Then there is a natural  decreasing, multiplicative $\mathbb{Z}_{\geq 0}^{op}$-indexed filtration
 $\tN^{\geq \ast}\frac{\Prismbar_R}{p}$ on $\frac{\Prismbar_R}{p}$ with associated graded
given as 
\begin{equation} 
\mathrm{gr}^{\ast} \frac{\Prismbar_R}{p}
\simeq \left( \bigoplus_{i \geq 0} \mathcal{N}^i \frac{\Prism_R}{p}
\right)/\theta,
\end{equation} 
where $\theta$ lives in grading $p$. Furthermore, for any $i \in \mathbb{Z}$, we can construct a
similar filtration $\tN^{\geq \ast}\frac{\Prism_R\left\{i\right\}}{p}$, which is a module
over the filtration on $\frac{\Prism_R}{p}$; the associated graded terms are
given individually as
\begin{equation} \label{grTwistedNygaard} \mathrm{gr}^j  \frac{
\Prismbar_R\left\{i\right\}}{p} \simeq \mathrm{cofib}\left( \theta: \mathcal{N}^{j-p}
\frac{\Prism_R}{p} \to \mathcal{N}^j \frac{\Prism_R}{p} \right), \end{equation}
where $\mathcal{N}^{j} \frac{\Prism_R}{p} = 0$ for $j < 0$.
In fact, by descent from $\qrsp_{\mathbb{Z}_p}$ and left Kan extension, these
claims 
follow from 
\Cref{NygaardHT} combined with 
the identification of \Cref{identificationofp}. 
\end{construction}

\begin{remark} 
\label{whentwistedcomplete}
The twisted Nygaard filtration $\tN^{\geq \ast}
\frac{\Prismbar_R\left\{i\right\}}{p}$ is complete if and only if the Nygaard
filtration $\mathcal{N}^{\geq \ast}\Prismbar_R\left\{i\right\}$ is complete, as
follows by $p$-completeness. 

\end{remark}

\subsection{Relative perfectness}
In the sequel we will study how the above filtration varies as $R$ does. 
To begin, for future reference we include here a special case of this result based on
the notion of \emph{relative perfectness}. 

\begin{definition}[Relatively perfect maps] 
Let $A$ be an animated ring, and let $B$ be an animated $A$-algebra.
We say that $B$ is \emph{relatively perfect} over $A$ if the diagram
\[ \xymatrix{
A/p  \ar[d]^{\phi}  \ar[r] & B/p  \ar[d]^{\phi} \\
A/p \ar[r] &  B/p
}\]
is a pushout square of animated rings.
This implies that the cotangent complex $L_{B/A}$ vanishes $p$-adically,
cf.~\cite[Cor.~3.8]{Bhattpadic}, so $L_{A/\mathbb{Z}} \otimes_A B \to
L_{B/\mathbb{Z}}$ is a $p$-adic equivalence. 
\end{definition} 

\begin{remark} 
Suppose $A, B$ are discrete rings and $A \to B$ is $p$-completely flat. Then 
$A \to B$ is relatively perfect in the above sense if and only if the analogous
diagram 
involving the \emph{ordinary} quotients of $A, B$ by $(p)$ is cocartesian. 
In fact, we claim that if $R \to S$ is any flat
map of animated $\mathbb{F}_p$-algebras, then $R \to S$ is relatively perfect in
the animated sense if and only if $\pi_0(R) \to \pi_0(S)$ is relatively perfect in
the classical sense. The ``only if'' direction is clear as applying $\pi_0(-)$ preserves pushout squares. For the reverse implication, observe that $R \to S$ is relatively perfect
in the animated sense exactly when the relative Frobenius $(S/R)^{(1)} := S
\otimes_{R,\varphi} R \to S$ is an isomorphism of animated $R$-algebras. Now
base change along $R \to \pi_0(R)$ is conservative on connective $R$-modules, so
it suffices to check that $(S/R)^{(1)} \otimes_R \pi_0(R) \to S \otimes_R
\pi_0(R)$ is an isomorphism in $\mathcal{D}(\pi_0(R))$. Noting that the formation of the relative Frobenius commutes with arbitrary base change along maps of animated rings, it remains to observe that $\pi_0(R) \to \pi_0(S)$ identifies with the base change $\pi_0(R) \to S \otimes_R^L \pi_0(R)$ of $R \to S$ by the flatness assumption, and that the Frobenius twist of a flat $\pi_0(R)$-algebra is automatically discrete.
\end{remark} 

\begin{proposition} 
\label{relativelyperfectbasechange}
Let $A \to B$ be a relatively perfect map of animated rings. 
Then the natural map induces an equivalence (after $p$-completion) of filtered objects 
$\mathcal{N}^{\geq \ast} \Prismbar_A\left\{i\right\} \otimes_A B \simeq
\mathcal{N}^{\geq \ast} \Prismbar_B\left\{i\right\}$, and similarly for the
twisted Nygaard filtrations on $\frac{\Prismbar_{(-)}\left\{i\right\}}{p}$. 
Moreover, for each $i$, we have a $p$-adic equivalence
$\mathcal{N}^i \Prism_A \otimes_A B \simeq \mathcal{N}^i \Prism_B$. 
\end{proposition} 
\begin{proof} 
We have that
$\Prismbar_A\left\{i\right\} \otimes_A B \to \Prismbar_B\left\{i\right\}$ is an
equivalence by the $p$-complete vanishing of the cotangent complex, e.g., by
comparing the
absolute conjugate filtrations, \cite[Sec.~4.5]{BhattLurieAPC}. 
This also yields the claim about the Nygaard pieces $\mathcal{N}^i \Prism$,
using the Nygaard fiber sequence
\cite[Rem.~5.5.8]{BhattLurieAPC}. 
Finally, the claim about 
$\mathcal{N}^{\geq \ast} \Prismbar$ now follows from the claims about
$\Prismbar$ and $\mathcal{N}^i \Prism$; note that we need relative perfectness
and not only $p$-adic vanishing of the relative cotangent complex because of the
restrictions along Frobenius involved in 
\eqref{grNygaardHT}. 
\end{proof} 

\subsection{Polynomial rings}
The purpose of this subsection is to identify explicitly the Hodge--Tate
cohomology of a polynomial ring, together with its Nygaard filtration
(\Cref{prismbarpolyHT}). We also treat the easier case of the Nygaard graded pieces
of prismatic cohomology (\Cref{keycofibpolypieces}). 

In the sequel, we use the following. Let $\left\{A^{\geq \ast}\right\}$ be a
filtered ring. Then the $\infty$-category $\mathcal{D}(A^{\geq \ast})$ of
$A^{\geq \ast}$-modules in the filtered derived $\infty$-category admits a
$t$-structure, where (co)connectivity is checked levelwise, and such that the
heart consists of modules over $A^{\geq \ast}$ in the category
$\mathrm{Fun}(\mathbb{Z}^{op}, \mathrm{Ab})$; we will sometimes simply refer to
these as $A^{\geq \ast}$-modules. 

In addition, for future reference, it will be helpful to keep track of the
naturally arising internal gradings that arise, which we first review.

\begin{remark}[Automatic internal gradings] 
\label{autointernalgrading}
Let $\mathcal{F}$ be a functor from $\qrsp_{\mathbb{Z}_p}$ to $p$-complete
abelian groups with (for simplicity) bounded $p$-power torsion. 
Suppose that for any $R \in \qrsp_{\mathbb{Z}_p}$, we are given an $R$-module structure on 
$\mathcal{F}(R)$ which is natural in $R$ in the evident sense. Suppose further
that for any such $R$, the natural map 
$\mathcal{F}(R) \otimes_R R[t^{1/p^\infty}] \to \mathcal{F}( R\left \langle
t^{1/p^\infty}\right\rangle)$ is a $p$-adic equivalence.

Then for any $R' \in
\qrsp_{\mathbb{Z}_p}$ with a $\mathbb{Z}[1/p]_{\geq 0}$-grading (in the
$p$-complete sense),  the $R'$-module
$\mathcal{F}(R')$ also inherits a canonical $\mathbb{Z}[1/p]_{\geq 0}$-grading
for essentially diagrammatic reasons. 
We have a map $\mathrm{coact}: R' \to R'\left \langle t^{1/p^\infty}\right\rangle$ carrying a
homogeneous element $x \in R'_i$ to $x \otimes t^i$. 
An element $y \in \mathcal{F}(R')$ is homogeneous of degree $i
\in  \mathbb{Z}[1/p]_{\geq 0}$ if and only if it maps under $\mathrm{coact}$
to $y \otimes t^{i} \in \mathcal{F}(R'\left \langle t^{1/p^\infty}\right\rangle)
\simeq \mathcal{F}(R')\left \langle t^{1/p^\infty}\right\rangle$. 
\end{remark} 

\begin{construction}[Internal gradings on Hodge--Tate cohomology] 
Let $R$ be a $\mathbb{Z}[1/p]_{\geq 0}$-graded animated ring. 
In this case, the (twisted) Hodge--Tate cohomology together with its Nygaard
filtration 
$\mathcal{N}^{\geq \ast}\Prismbar_R\left\{i\right\}$ naturally inherits the
structure of a $\mathbb{Z}[1/p]_{\geq 0}$-graded object of
$\widehat{\mathcal{D}(\mathbb{Z}_p)}$. 
Explicitly, one uses quasisyntomic descent, animation,
\Cref{autointernalgrading}, and that 
the natural map 
\[ \mathcal{N}^{\geq \ast}\Prismbar_R\left\{i\right\} \otimes_\mathbb{Z}
\mathbb{Z}[t^{1/p^\infty}] \to \mathcal{N}^{\geq \ast }\Prismbar_{R
\otimes_{\mathbb{Z}} \mathbb{Z}[t^{1/p^\infty}]}\left\{i\right\}  \]
is an isomorphism $p$-adically by relative perfectness
(\Cref{relativelyperfectbasechange}).\footnote{In the language of \cite{BhattLurieprismatization}, the Hodge--Tate stack
associated to the
group scheme $\mathbb{G}_m^{\mathrm{perf}} = \varprojlim_p \mathbb{G}_m$ is 
$\mathbb{G}_m^{\mathrm{perf}} \times \wcartHT$ by relative perfectness, so if a scheme $X$ is equipped with a
$\mathbb{G}_m^{\mathrm{perf}}$-action, then so is its Hodge--Tate
stack $\wcartHT_X$.} Similarly, in the above setting, \Cref{autointernalgrading} yields an additional grading
on 
$\mathcal{N}^i \Prism_R, i \geq 0$. 
Since there will be multiple gradings at the same time, we will refer to these
internal gradings as weight gradings. 
\end{construction} 

\begin{remark} 
Let $R$ be a $\mathbb{Z}[1/p]_{\geq 0}$-graded ring. 
If $R$ is concentrated in degrees $\mathbb{Z}_{\geq 0}$, then
$\mathcal{N}^i \Prism_R$ and 
$\Prismbar_R\left\{i\right\}$ are concentrated in degrees $\mathbb{Z}_{\geq 0}$,
as one sees using the conjugate filtration over a perfectoid base.  
However, the associated graded terms of the Nygaard filtration are in degrees
$\frac{1}{p} \mathbb{Z}_{\geq 0}$: this follows from \eqref{grNygaardHT} 
noting that there is a restriction along Frobenius involved, which divides
degrees by $p$. 
\end{remark}

\begin{proposition} 
Let $R$ be a $p$-torsionfree quasiregular semiperfectoid ring. Then there are natural
isomorphisms of graded $A^{\ast} = \bigoplus_{i \geq 0} \widehat{\mathcal{N}^i \Prism_R
\otimes_R R[x]}$-modules
\begin{equation} 
H^j( \bigoplus_{i \geq 0} \mathcal{N}^i \Prism_{R[x]}) \simeq 
\begin{cases} 
 A^{\ast}, \quad j = 0 \\
A^{\ast -1} , \quad j = 1 
 \end{cases} 
\end{equation}
With respect to the internal weight grading with $|x| = 1$ and $R$ in weight
zero, then the generator in
$H^0$ has weight zero and the generator in $H^1$ has weight $1$. 
 \label{keycofibpolypieces}
\end{proposition} 
\begin{proof} 
The generator in $H^0$ is simply the unit. 
The generator in $H^1(\mathcal{N}^1 \Prism_{R[x]})$ comes from the class
$dx$, via the isomorphism $\mathcal{N}^1 \Prism_S \simeq
\widehat{L_{S/\mathbb{Z}}}[-1]$ for any animated ring $S$,
cf.~\cite[Prop.~5.5.12]{BhattLurieAPC}. 
Having named the classes, it suffices by base-change 
(since for any perfectoid ring $R_0$, the functor $R \mapsto \bigoplus_{i \geq 0} \mathcal{N}^i \Prism_R$ is a symmetric
monoidal functor from animated $R_0$-algebras to $p$-complete graded objects)
to verify the isomorphism 
when $R$ is perfectoid, where the result follows from 
the isomorphisms with the conjugate filtration: for any $R$-algebra $S$ (in
particular, $R[x]$),
$\mathcal{N}^i \Prism_S \simeq \mathrm{Fil}_{\leq i}^{\mathrm{conj}}
\Prismbar_S$ by 
\cite[Th.~12.2]{Prisms}, and using the Hodge--Tate comparison for the latter
\cite[Th.~4.11]{Prisms}. 
\end{proof}

\begin{proposition} 
\label{prismbarpolyHT}
Let $R \in 
\qrsp_{\mathbb{Z}_p}
$ be a $p$-torsionfree quasiregular semiperfectoid ring. 
Let $A^{\geq \ast}$ be the $p$-completion of 
the filtered ring
$\mathcal{N}^{\geq \ast}
\Prismbar_R \otimes_R R[x]$. 
Then there are isomorphisms of $A^{\geq \ast}$-modules
\begin{equation}
H^0( \mathcal{N}^{\geq \ast}\Prismbar_{R[x]}) \simeq A^{\geq \ast}, \quad H^1(
\mathcal{N}^{\geq \ast}\Prismbar_{R[x]}) \simeq A^{\geq \ast-1}\left\{-1\right\} \oplus
\bigoplus_{i=1}^{p-1} A^{\geq \ast-1}/A^{\geq \ast}.
\end{equation}
With respect to the internal weight grading with $|x| = 1$ and $R$ in degree
zero, the generator of
$H^0$ is in weight zero, the generator of $A^{\geq \ast -1}\left\{-1\right\}$ is
in weight one, and the $i$th 
copy of $A^{\geq \ast-1}/A^{\geq \ast}$ has generator in weight $\frac{i}{p}$. 
\end{proposition} 
\begin{proof} 
Let us first name the generators. The generator of $H^0$ is simply $1$. 
The first generator in $H^1$
is the class $dx \in H^1(  \Prismbar_{R[x]}\left\{1\right\})$
constructed via the boundary map 
$\Prismbar_{R[x]}\left\{1\right\} \to \Prism_{R[x]}/I^2 \to \Prismbar_{R[x]}$ as
the image of $x$ (note that this boundary map is how one produces the Hodge--Tate
comparison, \cite[Cons.~4.9]{Prisms}); it lifts uniquely to $H^1( \mathcal{N}^{\geq 1}
\Prismbar_{R[x]}\left\{1\right\})$, and thus produces a map of $A^{\geq
\ast}$-modules
$A^{\geq \ast -1}\left\{-1\right\} \to H^1( \mathcal{N}^{\geq \ast}
\Prismbar_{R[x]})$. 
Next, we have the fiber sequence of $R[x]$-modules
\[ \mathcal{N}^{\geq 1} \Prismbar_{R[x]} \to \Prismbar_{R[x]} \to
R/p^{(-1)}[x^{1/p}],  \]
from the description \eqref{gr0prismbar} (and quasisyntomic descent) to
identify $\mathrm{gr}^0 \Prismbar_{R[x]} = R/p^{(-1)}[x^{1/p}]$. 
For each $0 < i < p$, the boundary map applied to $x^{i/p}$ gives a class in
$H^1( \mathcal{N}^{\geq 1} \Prismbar_{R[x]})$ of weight $i/p$; by construction,
this class is annihilated by $A^{\geq 1}$ since $R/p^{(-1)}[x^{1/p}]$ is by
definition, whence we obtain 
maps in from $A^{\geq \ast-1}/A^{\geq \ast}$. 

Since we have named the generating classes, 
to prove the isomorphism, we may assume (by base-change) that $R  $
is a $p$-torsionfree \emph{perfectoid} ring.\footnote{Let $R_0$ be a
$p$-torsionfree perfectoid
ring. Then the
construction 
$R \mapsto \mathcal{N}^{\geq \ast} \Prismbar_R$, from animated $R_0$-algebras to 
$p$-complete filtered $\mathcal{N}^{\geq \ast} \Prismbar_{R_0}$-algebras,
preserves colimits, and in particular preserves coproducts. In fact, this holds for $R \mapsto \Prismbar_R$ itself by 
the Hodge--Tate comparison, and on associated graded terms by 
\eqref{grNygaardHT} and \cite[Th.~12.2]{Prisms}. 
} 
Moreover, by descent in $R$, we may assume that $R$ contains a $p$th root of
$p$, e.g., using Andr\'e's lemma in the form of \cite[Th.~7.14]{Prisms}. 
We make this assumption for the rest of the argument.  This implies that the
Nygaard filtration on $R \simeq \Prismbar_R$ is the filtration by powers of
$p^{1/p}$, cf.~\Cref{identificationofp} and \eqref{gr0prismbar}. 

In this case, we have isomorphisms
(via the Hodge--Tate comparison \cite[Th.~4.11]{Prisms})
\[ H^i( \Prismbar_{R[x]}) 
\simeq \begin{cases} 
R\left \langle x\right\rangle, \quad i = 0 \\
R \left\{-1\right\} \left \langle x\right\rangle dx, \quad i = 1,
 \end{cases} 
\]
where the class $dx$ arises from the image of the class $x$ under the connecting
map in the cofiber sequence
$\Prismbar_{R[x]}\left\{1\right\} \to \Prism_{R[x]}/I^2
\to \Prismbar_{R[x]}$. 

Using the expression \eqref{grNygaardHT} for the Nygaard filtration (which is complete in this
case since the algebra is smooth over a perfectoid, so we can check on
associated graded terms), we find that 
multiplication by $p^{1/p}$ induces isomorphisms 
$p^{1/p}: \mathcal{N}^{\geq i} \Prismbar_{R[x]} \simeq
\mathcal{N}^{\geq i+1}
\Prismbar_{R[x]}$ for $i > 0$, also using the comparison between the
associated graded pieces of the Nygaard filtration and the Hodge--Tate
filtration \cite[Th.~12.2]{Prisms}.  
As above, we can identify the map $R \left \langle x\right\rangle\to
\Prismbar_{R[x]} \to
\mathrm{gr}^0 \Prismbar_{R[x]}$ with 
the $R$-linear map 
$R\left \langle x\right\rangle \to R/p^{1/p} \left
\langle x^{1/p}\right\rangle = (R/p\left \langle
x\right\rangle)^{(-1)}, x \mapsto x$ (unwinding the restriction
along Frobenius as in \eqref{gr0prismbar}). 
This yields
\begin{equation} H^\ast(\mathcal{N}^{\geq 1}
\Prismbar_{R[x]}) \simeq \begin{cases} 
p^{1/p} R\left \langle x\right\rangle, \quad \ast = 0 \\
R\left\{-1\right\} \left \langle x\right\rangle dx  \oplus
{\bigoplus}_{i
\geq 0, p \nmid i}R/p^{1/p} \cdot
x^{i/p}, \quad \ast = 1.\\
\end{cases} \label{polynomialringN}\end{equation} 
It follows that, as filtered $A^{\geq \ast}
\stackrel{\mathrm{def}}{=}\widehat{\mathcal{N}^{\geq \ast}
\Prismbar_{R} \otimes_{R} R[x]}$-modules in
$\mathrm{Fun}(\mathbb{Z}_{\geq 0}^{op}, \mathrm{Ab})$, the classes specified
yield a natural isomorphism
\begin{equation} 
\label{H1Npoly2}
H^1( \mathcal{N}^{\geq \ast} \Prismbar_{R[x]}) \simeq
A^{\geq \ast}\left\{-1\right\} \oplus 
\bigoplus_{i = 1}^{p-1} A^{\geq \ast-1}/ A^{\geq \ast} . \qedhere
\end{equation} 
\end{proof}

\subsection{The Hodge--Tate cohomology of a quotient}
In this subsection, we use the results of the previous subsection on polynomial
rings to get an expression (via a fiber sequence) of the Hodge--Tate cohomology of a
quotient (\Cref{prismHTfibseq}), and some control of the Nygaard filtration too
(\Cref{thicksubcatlemma2}). 
To begin we start with the (easier) case of the Nygaard pieces themselves.

\begin{proposition} 
\label{Nygaardpiecespolynomialquotient}
Let $R$ be any animated $\mathbb{Z}[x]$-algebra. 
Then there exists a natural fiber sequence of graded 
$\bigoplus_{i \geq 0} \mathcal{N}^i \Prism_R$-modules 
\begin{equation}
\left(\bigoplus_{i \geq 0} \mathcal{N}^i \Prism_R  \right)/x\to 
\bigoplus_{i \geq 0} \mathcal{N}^i \Prism_{R/x}
\to \bigoplus_{i \geq 0} \mathcal{N}^{i-1} \Prism_{R/x}. 
\label{Nygaardpieceseqpoly}
\end{equation} 
\end{proposition} 
\begin{proof} 
First, let $B \in \qrsp_{\mathbb{Z}_p}$. 
We construct a cofiber sequence, naturally in $B$,
of $\bigoplus_{i \geq 0} \mathcal{N}^i \Prism_{B[x]}/x$-modules
\begin{equation} \label{weightgrfilt}
\bigoplus_{i \geq 0} \mathcal{N}^{i-1} \Prism_B [-1] \to 
\bigoplus_{i \geq 0} \mathcal{N}^i \Prism_{B[x]}/x \to \bigoplus_{i \geq 0}
\mathcal{N}^i \Prism_{B}
.\end{equation}
To construct this, we use \Cref{keycofibpolypieces}, 
which shows that 
the (bi)graded $E_\infty$-ring
$\bigoplus_{i \geq 0} \mathcal{N}^i \Prism_{B[x]}/x$ is concentrated in weights
$0$ and $1$, using the weight grading on $B[x]$ with $|x| =1$ and $B$ in weight
zero. Now any weight-graded module 
over
$\bigoplus_{i \geq 0} \mathcal{N}^i \Prism_{B[x]}/x$
admits a filtration by the weight
grading, which gives the cofiber sequence \eqref{weightgrfilt}, using again 
\Cref{keycofibpolypieces} to identify the weight zero and weight one components
with 
$\bigoplus_{i \geq 0} \mathcal{N}^i \Prism_B$ and $\bigoplus_{i \geq 0}
\mathcal{N}^{i-1} \Prism_B$. 

By base-change and descent, one now deduces the proposition.  In fact, we may
assume that $R$ is an $B[x]$-algebra  for some $B \in \qrsp_{\mathbb{Z}_p}$,
provided everything is done independently of the choice of $B$. 
Then the desired 
\eqref{Nygaardpieceseqpoly} follows from \eqref{weightgrfilt}, using that
\[ \bigoplus_{i \geq 0} \mathcal{N}^i \Prism_R \otimes_{\bigoplus_{i \geq 0}
\mathcal{N}^i\Prism_{B[x]}} \bigoplus_{i \geq 0} \mathcal{N}^i\Prism_B  \to \bigoplus_{i \geq 0}
\mathcal{N}^i\Prism_{R/x} \]
is a $p$-adic equivalence. 
\end{proof}

\begin{proposition} 
\label{thicksubcatlemma}
Let $B \in \qrsp_{\mathbb{Z}_p}$. Then, naturally in $B$, there is a finite
filtration on $\mathcal{N}^{\geq \ast}
\Prismbar_{B[x]} / x$  in 
$\mathcal{N}^{\geq \ast}
\Prismbar_{B[x]} / x$-modules 
 whose associated graded terms are 
$\mathcal{N}^{\geq \ast} \Prismbar_B$, 
$(p-1)$ copies of 
$\left(\mathcal{N}^{\geq \ast -1}\Prismbar_B/\mathcal{N}^{\geq \ast}\Prismbar_B\right)[-1]$, and 
$\mathcal{N}^{\geq \ast-1}\Prismbar_B\left\{-1\right\}[-1]$. 
\end{proposition} 
\begin{proof} 
In fact, this follows from the natural expression 
\eqref{prismbarpolyHT}, noting the weight grading (with $|x| = 1$). In
particular, 
$\mathcal{N}^{\geq \ast}
\Prismbar_{B[x]} / x$ has weights in $0, \frac{1}{p}, \frac{2}{p}, \dots, 1$
with the weight zero component being $\mathcal{N}^{\geq \ast} \Prismbar_B$, the
weight $\frac{i}{p}$ component for $0 < i < p$ being 
$\left(\mathcal{N}^{\geq \ast -1}\Prismbar_B/\mathcal{N}^{\geq \ast}\Prismbar_B\right)[-1]$, and the weight
$1$ component being $\mathcal{N}^{\geq \ast-1}\Prismbar_B\left\{-1\right\}[-1]$. 
\end{proof}

\begin{corollary} 
\label{thicksubcatlemma2}
Let $A$ be any animated $\mathbb{Z}[x]$-algebra. 
Then the filtered object $\mathcal{N}^{\geq \ast} \Prismbar_{A}/x\left\{i\right\}$ 
admits a natural finite filtration, whose associated graded terms
are 
$\mathcal{N}^{\geq \ast} \Prismbar_{A/x}\left\{i\right\}$, 
then $(p-1)$ copies of
$\left(\mathcal{N}^{\geq \ast -1}\Prismbar_{A/x}/\mathcal{N}^{\geq
\ast}\Prismbar_{A/x}\right)\left\{i\right\}[-1]$, and 
$\mathcal{N}^{\geq \ast-1}\Prismbar_{A/x}\left\{i-1\right\}[-1]$. 
\end{corollary} 
\begin{proof}
It suffices to replace $\mathbb{Z}[x]$ by $B[x]$ for $ B \in
\qrsp_{\mathbb{Z}_p}$ and construct the filtration naturally in $B$, by
quasisyntomic descent. But then the claim follows from \Cref{thicksubcatlemma}. 
\end{proof} 

We separately record the resulting fiber sequence for Hodge--Tate cohomology
itself (forgetting the Nygaard filtration in \Cref{thicksubcatlemma2}).  Such a fiber sequence can also be produced using the
description of the Hodge--Tate stack of the affine line,
cf.~\cite[Ex.~9.1]{BhattLurieprismatization}. 
\begin{corollary} 
\label{prismHTfibseq}
Let $R$ be any animated $\mathbb{Z}[x]$-algebra. Then there is a natural fiber
sequence
\begin{equation}  \label{naturalfibseq} \Prismbar_{R}/x\left\{i\right\} \to \Prismbar_{R/x}\left\{i\right\}  \to
\Prismbar_{R/x}\left\{i-1\right\}. \end{equation} 
\end{corollary}

\section{$F$-smoothness}

The goal of this section is to formulate the notion of $F$-smoothness
(\Cref{filtSegaldef}). This is a variant of ($p$-adic) smoothness designed to capture
smoothness in an absolute sense. For instance, smooth algebras over any
perfectoid ring are $F$-smooth (\Cref{FSmoothPerfd}), as are regular rings
(\Cref{regularringsSegal}); in fact, the latter is the main result of this
section. Our idea is to essentially define $F$-smoothness by demanding a strong
form of the $L\eta$-isomorphism in relative prismatic cohomology (\cite[Theorem
15.3]{Prisms}, \cite{BMS1}), adapted to the absolute prismatic context using the
Beilinson $t$-structure interpretation of the $L\eta$ functor as in \cite[\S
5.1]{BMS2} (see \Cref{FSmoothBeilinson}). To work effectively with this notion,
we need access to the certain naturally defined elements of the prismatic
cohomology (or variants) of $\mathbb{Z}_p$ introduced in \S\ref{sec:calculations}.

\subsection{$F$-smoothness: definition}

Let $A$ be a $p$-quasisyntomic ring. 
Recall \cite[Def.~4.1]{BMS2} that an object $M \in \mathcal{D}(A)$ has
\emph{$p$-complete $\mathrm{Tor}$-amplitude} in  degrees $\geq r$ if for every
discrete $A/p$-module $N$, we have $M \otimes_A^L N \in \mathcal{D}^{\geq r}(A)$. 
\begin{definition}[$F$-smoothness]
\label{filtSegaldef}
We say that $A$ \emph{is $F$-smooth} if for each $i \in \mathbb{Z}_{\geq 0}$,
the map in $\mathcal{D}(A)$,
\[ \phi_i: \mathcal{N}^i \Prism_A \to {\Prismbar_A}\left\{i\right\}  \]
induced by the Frobenius on $\Prism_A$  has fiber $\mathrm{fib}(\phi_i ) $ with $p$-complete $\mathrm{Tor}$-amplitude
in  degrees $\geq i+2$ and if the Nygaard filtration on
$\Prismbar_A\left\{i\right\}$ (or
equivalently $\Prism_A\left\{i\right\}$) is complete. 
Note that this condition only depends on the $p$-completion of $A$. 

  We say that a $p$-quasisyntomic scheme is $F$-smooth  if it is covered by the
  spectra of rings which are $F$-smooth (note that $F$-smoothness is preserved by Zariski localization by \Cref{FSmoothColimitEtale} below).
  \end{definition}

The condition of Nygaard-completeness in the definition of $F$-smoothness is
slightly delicate. 
In order to work with it, we  will also use the following auxiliary condition. 
\begin{definition}[Weak $F$-smoothness]
\label{DefWeakFSmooth}
We say that a $p$-quasisyntomic ring $A$ is {\em weakly $F$-smooth} if for each $i$, the object
\begin{equation} \mathrm{fib}\left(  \theta: \mathcal{N}^{i} \frac{\Prism_A}{p}
\to \mathcal{N}^{i+p} \frac{\Prism_A}{p} \right) \in \mathcal{D}(A) ,\end{equation}
has $p$-complete $\mathrm{Tor}$-amplitude in degrees $\geq i+1$.  If $A$ is $p$-torsionfree and weakly $F$-smooth, then the above fiber is concentrated in degrees $\geq i+2$, as it is $p$-torsion. 
\end{definition}
\begin{proposition}[$F$-smoothness vs weak $F$-smoothness]
\label{FSmoothWeak1}
If a $p$-quasisyntomic ring $A$ is $F$-smooth,  then $A$ is weakly $F$-smooth.
Conversely, the $p$-quasisyntomic ring $A$ is $F$-smooth if and only it is weakly $F$-smooth and the natural map
of graded $E_\infty$-rings
\begin{equation}  \bigoplus_i \phi_i :  \bigoplus_{i \geq 0} \mathcal{N}^i \frac{\Prism_A}{p} \to 
\bigoplus_{i \in \mathbb{Z}} \frac{ \widehat{\Prismbar}_A\left\{i\right\}}{p}
\label{needtoinverttheta}
\end{equation} 
(where the target denotes the direct sum of the \emph{Nygaard-completed}
Hodge--Tate cohomologies mod $p$)
exhibits the target as the localization of the source at $\theta$. 
\end{proposition}
\begin{proof} 
The first claim follows from the commutative diagram
\[ \xymatrix{
\mathcal{N}^i \frac{\Prism_A}{p} \ar[d]^{\phi_i}  \ar[r]^{\theta} & \mathcal{N}^{i+p}
\frac{\Prism_A}{p} \ar[d]^{\phi_{i+p}}  \\
\frac{\Prismbar_A\left\{i\right\}}{p} \ar[r]^{\theta, \simeq} &
\frac{\Prismbar_A\left\{i+p\right\}}{p} 
}\]
obtained from the map of graded 
$E_\infty$-rings $\bigoplus_{i \geq 0} \mathcal{N}^i \frac{\Prism_A}{p} \to
\bigoplus_{i \in \mathbb{Z}}  \frac{\Prismbar_A\left\{i\right\}}{p}$. 
The second claim follows from the above and \Cref{invertingtheta}: the localization of the
source in \eqref{needtoinverttheta} is precisely the mod $p$ Hodge--Tate
cohomology. 
\end{proof} 

\begin{remark}[Stability of weak $F$-smoothness under filtered colimits and \'etale localization]
\label{WeakFSmoothColimitEtale}
As the construction $A \mapsto \mathcal{N}^i \Prism_A \in \mathcal{D}(A)$ commutes with
$p$-completed filtered colimits and  \'etale localization, it follows that the
collection of weakly $F$-smooth rings is closed under filtered colimits and
 \'etale localizations inside all $p$-quasisyntomic rings. Moreover, weak $F$-smoothness can be detected locally for the  \'etale topology.
\end{remark}

\begin{remark}[Essential constancy of the twisted Nygaard filtration under weak $F$-smoothness]
\label{EssConstTwistNyg}
If $A$ is weakly $F$-smooth, then for any fixed integer $n$, we have
\[ H^n\left(\mathrm{fib}\left(  \theta: \mathcal{N}^{j} \frac{\Prism_A}{p} \to \mathcal{N}^{j+p} \frac{\Prism_A}{p} \right)\right) = 0 \quad \text{for} \quad j \gg 0.\]
It follows that the twisted Nygaard filtration on
$\frac{\Prismbar_A\left\{i\right\}}{p}$ (\Cref{TwistedNygaard}) is essentially constant in each cohomological degree; moreover, the implicit constants are independent of $A$.
\end{remark}

\begin{proposition}[Stability of $F$-smoothness under filtered colimits and \'etale localization]
\label{FSmoothColimitEtale}
The property of being $F$-smooth is stable under filtered colimits.
\end{proposition}
\begin{proof} 
 Given a filtered diagram $\{A_i\}$ of $F$-smooth rings with 
 colimit $A$, each $A_i$ is weakly $F$-smooth by \Cref{FSmoothWeak1}; it then
 follows from \Cref{EssConstTwistNyg} that the $p$-completion of $\varinjlim_i
 \widehat{\Prismbar}_{A_i}\{j\}$ gives $\widehat{\Prismbar}_{A}\{j\}$, which easily shows that $A$ is $F$-smooth. \end{proof}

For the next result, cf.~also \cite[Prop.~9.5.11]{BLM} for the analog in
characteristic $p$.  Unlike in \emph{loc.~cit.}, we make a ($p$-complete) flatness hypothesis; we
expect that this should be unnecessary, but were unable to remove it. 
\begin{proposition} 
\label{FrobetalefilteredSegal}
Let $A$ be a $p$-quasisyntomic ring. Let $B$ be a $p$-completely flat
$A$-algebra which is relatively perfect. If $A$ is $F$-smooth, so is $B$. Moreover, the converse holds true if $B$ is $p$-completely faithfully flat over $A$.
In particular, $F$-smoothness is \'etale local and passes to \'etale algebras. 
\end{proposition}
\begin{proof}
We have $p$-adic equivalences  $\mathcal{N}^i \Prism_A \otimes_A B \xrightarrow{\sim} \mathcal{N}^i
\Prism_B$ by \Cref{relativelyperfectbasechange}. From this, it follows that if $A$ is weakly $F$-smooth, then so is $B$; the
converse holds if $B$ is $p$-completely faithfully flat over $A$.

Next, we have that
$\tN^{\geq \ast}\frac{\Prismbar_A\left\{i\right\}}{p} \otimes_A B \to
\tN^{\geq \ast}\frac{\Prismbar_B\left\{i\right\}}{p}$ 
is an equivalence,
again by \Cref{relativelyperfectbasechange}. 
If $A$ is $F$-smooth, then 
$\tN^{\geq \ast}\frac{\Prismbar_A\left\{i\right\}}{p} $ is pro-zero in any range
of degrees, whence the same holds true for 
$\tN^{\geq \ast}\frac{\Prismbar_B\left\{i\right\}}{p} $ (by $p$-complete
flatness), whence completeness of the Nygaard filtration
(\Cref{whentwistedcomplete}); we conclude $B$ is
then $F$-smooth. 
The converse if $B$ is $p$-completely faithfully flat follows similarly. 
\end{proof}

\begin{proposition} 
\label{laurentpolySegal}
If a $p$-quasisyntomic ring $A$  is $F$-smooth, then  the polynomial ring ${A[x]}$ is also $F$-smooth.
\end{proposition} 
\begin{proof} 
Suppose $A$ is $F$-smooth. 
The weak $F$-smoothness of $A[x]$ follows 
using the cofiber sequence of $\bigoplus_{i \geq 0} \mathcal{N}^i
\Prism_A $-modules
obtained by unfolding \Cref{keycofibpolypieces},
\[ 
\left(\bigoplus_{i \geq 0} \mathcal{N}^{i-1} \Prism_A \widehat{\otimes_A} A[x]
\right)[-1] 
\to \bigoplus_{i \geq 0} \mathcal{N}^i \Prism_{A[x]} \to 
\bigoplus_{i \geq 0} \mathcal{N}^{i} \Prism_A \widehat{\otimes_A} A[x]	
.\]

By \Cref{prismbarpolyHT}, and quasisyntomic descent, we find that 
there is a finite filtration  
on $\mathcal{N}^{\geq \ast}\Prismbar_{A[x]}$ (considered as an object of the
filtered derived $\infty$-category) where the associated graded terms are given by  
the $p$-completions of 
${\mathcal{N}^{\geq \ast}\Prismbar_{A} \otimes_A A[x]}$, 
$(p-1)$ copies of 
${\mathcal{N}^{\geq \ast-1}\Prismbar_{A} \otimes_A A[x]}/
{\mathcal{N}^{\geq \ast}\Prismbar_{A} \otimes_A A[x]}$, and 
${\mathcal{N}^{\geq \ast-1}\Prismbar_{A} \otimes_A
A[x]}\left\{-1\right\}[-1]$. 
Thus, it suffices to show that 
under the $F$-smoothness hypotheses, 
$(\mathcal{N}^{\geq \ast} \Prismbar_A \otimes_A A[x])\left\{i\right\}$ is
complete mod $p$ for each $i \in \mathbb{Z}$. 
For this, it suffices to prove the analogous completeness with 
$\mathcal{N}^{\geq \ast} \Prismbar_A\left\{i\right\}$ replaced in the above
tensor product by the twisted
Nygaard filtration on the mod $p$ reduction (\Cref{TwistedNygaard}); however,
this follows from the essential 
constancy of the twisted Nygaard filtration,  \Cref{EssConstTwistNyg}. 
\end{proof} 

\begin{proposition} 
\label{formallysmoothsegal}
Let $A$ be a $p$-quasisyntomic ring, and let $B$ be a $p$-completely flat
$A$-algebra such that $A/p \to B/p$ is smooth. If $A$ is $F$-smooth, so is $B$. 
\end{proposition} 
\begin{proof} 
Combine \Cref{laurentpolySegal} and 
 \Cref{FSmoothColimitEtale}.
\end{proof}

\begin{proposition} 
\label{checkbylocalization}
Let $A$ be a $p$-quasisyntomic ring. Then $A$ is $F$-smooth if and only if all
the localizations $A_{\mathfrak{p}}$ for $ \mathfrak{p} \in \spec(A)$, are $F$-smooth. 
\end{proposition} 
\begin{proof} 
If $A$ is $F$-smooth, then all of its localizations are $F$-smooth by 
\Cref{FrobetalefilteredSegal}. 
The converse direction follows similarly as in the proof of 
\Cref{FrobetalefilteredSegal}, noting that $p$-complete
$\mathrm{Tor}$-amplitude can be checked on localizations. 
\end{proof}

\begin{remark}[$F$-smoothness and the Beilinson $t$-structure]
\label{FSmoothBeilinson}
Assume $A$ is an $F$-smooth $p$-quasisyntomic ring. Write $\Prism^{[\bullet]}_A$ for the complete filtered object defined by the prismatic complex $\Prism_A$ equipped with the filtration defined by powers of the Hodge--Tate ideal sheaf, so we have a natural identification
\[ \mathrm{gr}^* \Prism^{[\bullet]}_A \simeq \Prismbar_A\{\ast\}.\] 
By definition of the Nygaard filtration, the Frobenius on $\Prism_A$ refines to a map
\[ \varphi_A:\mathcal{N}^{\geq \ast} \Prism_A \to \Prism^{[\bullet]}_A\]
in the filtered derived category. Using the connectivity bound $\mathcal{N}^i
\Prism_A \in \mathcal{D}^{\leq i}(\mathbb{Z}_p)$  
(cf.~\cite[Rem.~5.5.9]{BhattLurieAPC}), the $F$-smoothness hypothesis implies in particular that $\varphi_A$ induces an equivalence
\[ \mathcal{N}^i \Prism_A \simeq \tau^{\leq i} \mathrm{gr}^i \Prism^{[\bullet]}_A \]
As both filtrations are complete by assumption, it follows that the map
$\varphi_A$ identifies its source with the connective cover of its target for
the Beilinson $t$-structure on the filtered derived category (see
\cite[Sec.~5.4]{BMS2} for an account). 
\end{remark}

\subsection{$F$-smoothness over a base}

In this subsection, we study the $F$-smoothness condition over a perfectoid
base. We offer the following characterization; work of V.~Bouis \cite{Bouis} has studied $F$-smoothness over mixed characteristic perfectoid base rings in more detail, and yielded important examples. 
\begin{proposition}[{Cf.~\cite[Th.~2.16, 2.18]{Bouis}}] 
\label{FSmoothPerfd}
Let $R_0$ be a perfectoid ring and let $A$ be an $R_0$-algebra.
Suppose $A$ is quasisyntomic. 
Then $A$ is $F$-smooth if and only if: 
\begin{enumerate}
\item $\widehat{L_{A/R_0}} $ is a $p$-completely flat $A$-module. 
\item The $p$-completed derived de Rham cohomology ${L \Omega_{A/R_0}}$
(cf.~\cite{Bhattpadic}) is
Hodge-complete. 
\end{enumerate}
\label{segaloverperfectoid}
\end{proposition} 
\begin{proof} 
The divided Frobenius $\phi_i: \mathcal{N}^i \Prismnc_A \to
\Prismbar_A$ (where we trivialize the Breuil--Kisin twists since we are
over $R_0$) matches the source with the $i$th stage of the conjugate filtration 
(cf.~\cite[Th.~4.11]{Prisms})
on the Hodge--Tate
cohomology, \cite[Th.~12.2]{Prisms}.

Now, the condition (2) that the $p$-completed derived de Rham cohomology is
Hodge-complete is equivalent to the condition that the derived prismatic
cohomology $\Prismnc_A$ (over the perfect prism corresponding to $R_0$) is
Nygaard-complete, thanks to \cite[Th.~7.2(5)]{BMS2}. 

Therefore, once one knows the derived prismatic cohomology is Nygaard-complete,
the $F$-smoothness condition amounts to the statement that 
the conjugate filtration map 
$\mathrm{Fil}_i \Prismbar_A \to \Prismbar_A$ has homotopy fiber
(in $\mathcal{D}(A)$) with $p$-complete $\mathrm{Tor}$-amplitude in degrees $\geq i+2$, for
each $i$. Using the associated gradeds of the conjugate filtration (given by
$\mathrm{gr}^j = \widehat{\wedge^j L_{A/R_0}}[-j]$), one easily
	sees by considering $i =0, 1$ that this is equivalent to the condition that $\widehat{L_{A/R_0}}$ should
be $p$-completely flat over $A$. 
\end{proof}

In the special case of quasisyntomic $\mathbb{F}_p$-algebras, the condition of $F$-smoothness had been previously studied under the name Cartier smoothness
\cite{KM21, KST21}
which we review next.\footnote{The second author had previously asked in \cite[Question 4.21]{survey} 
whether there could be a notion of Cartier smoothness in mixed characteristic;
we are also grateful to Matthew Morrow for discussions on this point.}

\begin{definition}[Cf.~\cite{KM21, KST21}] 
Let $A$ be a quasisyntomic $\mathbb{F}_p$-algebra. 
We say that $A$ is \emph{Cartier smooth} if: 
\begin{enumerate}
\item The cotangent complex $L_{A/\mathbb{F}_p}$ is a flat discrete $A$-module.  
\item The inverse Cartier map $C^{-1}: \Omega^{i}_{A/\mathbb{F}_p} \to H^i(
\Omega^{\ast}_{A/\mathbb{F}_p})$ is an isomorphism for $i \geq 0$. 
Here $\Omega^{\ast}_{A/\mathbb{F}_p}$ denotes the classical de Rham complex of
$A$ over $\mathbb{F}_p$. 
\end{enumerate}
\end{definition}

\begin{proposition} 
\label{FSmoothCartSmooth}
Let $A$ be a quasisyntomic $\mathbb{F}_p$-algebra. 
Then $A$ is $F$-smooth  if and only if
$A$ is {Cartier smooth}. 
\end{proposition} 
\begin{proof} 
Suppose $L_{A/\mathbb{F}_p}$ is a flat $A$-module. 
Then the derived de Rham cohomology $L \Omega_{A/\mathbb{F}_p}$ maps to its
Hodge completion, which is just the usual algebraic de Rham complex
$\Omega^\bullet_{A/\mathbb{F}_p}$. Using the conjugate filtration on 
the former \cite[Prop.~3.5]{Bhattpadic}, we see that 
the condition that this map should be an
equivalence is precisely the Cartier isomorphism condition. 
Therefore, the result follows from 
\Cref{segaloverperfectoid}. 
\end{proof}

\subsection{$F$-smoothness of regular rings}

In this subsection, we prove the following theorem.

\begin{theorem} 
\label{regularringsSegal}
Let $A$ be a regular (noetherian) ring. Then $A$ is $F$-smooth. 
Conversely, if $A$ is a $p$-complete noetherian ring which is $F$-smooth, then $A$ is regular. 
\end{theorem}

We first prove the forward direction. When $A$ is an $\mathbb{F}_p$-algebra,
$F$-smoothness is equivalently to Cartier smoothness (\Cref{FSmoothCartSmooth})
and thus follows at once from regularity via N\'eron--Popescu desingularization, which implies that $A$
is a filtered colimit of smooth $\mathbb{F}_p$-algebras. One can also prove the
result directly \cite[Sec.~9.5]{BLM}. 
In the case of an \emph{unramified} regular ring, most of the result also
appears in \cite[Prop.~5.7.9, 5.8.2]{BhattLurieAPC}.

\begin{proposition} 
\label{Fsmoothquotientregular}
Let $A$ be a $p$-quasisyntomic ring, and let $x \in A$ be a nonzerodivisor. 
Suppose $A/x$ and $A[1/x]$ are $F$-smooth. 
Then $A$ is $F$-smooth. 
\end{proposition} 

\begin{proof} 
First, we show that $A$ is weakly $F$-smooth. 
Write $\mathcal{F}^i_A = \mathrm{fib}( \theta: \mathcal{N}^i \frac{\Prism_A}{p}
\to \mathcal{N}^{i+p} \frac{\Prism_A}{p})$. 
Using the cofiber sequence of \Cref{Nygaardpiecespolynomialquotient},  we find that 
there is a cofiber sequence
$(\mathcal{F}^i_A ) /x \to \mathcal{F}^{i}_{A/x} \to \mathcal{F}^{i-1}_{A/x}$. 
Moreover, $\mathcal{F}^i_{A[1/x]} =(\mathcal{F}^i_A)[1/x]$. 
Note that 
an object $N \in \mathcal{D}(A)$ has $p$-complete $\mathrm{Tor}$-amplitude in degrees $\geq j$
if and only if $N[1/x] \in \mathcal{D}(A), N/x \in \mathcal{D}(A/x)$
have $p$-complete $\mathrm{Tor}$-amplitude in degrees $\geq j$. 
From these observations, it follows easily that $A$ is weakly $F$-smooth. 

Now we show that $A$ is $F$-smooth. 
For this, it suffices to show that the map
\begin{equation} \Prismbar_{A}\left\{i\right\} \to
\widehat{\Prismbar}_A\left\{i\right\} \label{compmapHTA} \end{equation} is an
equivalence for each $i$; here the latter denotes the Nygaard-completed
Hodge--Tate cohomology. 
By weak $F$-smoothness of $A$, the natural map 
$\widehat{\Prismbar}_A\left\{i\right\}[1/x] \to 
\widehat{\Prismbar}_{A[1/x]}\left\{i\right\}$
is an equivalence, thanks to \Cref{EssConstTwistNyg}. 
Therefore, by our assumptions, the comparison map \eqref{compmapHTA} becomes an
isomorphism after $p$-completely inverting $x$, so its fiber mod $p$ is $x$-power torsion. It thus
suffices to show that \eqref{compmapHTA} induces an isomorphism after
base-change along $A \to A/x$. But by \Cref{thicksubcatlemma2} and our assumption of
$F$-smoothness of $A/x$, the filtered object
$\mathcal{N}^{\geq \ast} \Prismbar_{A} / x$ is complete. 
\end{proof}

\begin{corollary} 
Let $A$ be a $p$-quasisyntomic ring such that $A$ is $p$-torsionfree and such
that the $\mathbb{F}_p$-algebra $A/p$ is Cartier smooth. Then $A$ is $F$-smooth. 
\end{corollary}
\begin{proof} 
Apply \Cref{Fsmoothquotientregular} with $x = p$. 
\end{proof}

\begin{proof}[Proof that regular rings are $F$-smooth] 
Suppose $A$ is regular. 
Since $A$ is lci,  $A$ is $p$-quasisyntomic. 
By 
\Cref{checkbylocalization}, the ring $A$ is  $F$-smooth if and only if all of
its localizations are $F$-smooth.  Consequently, we may assume that $A$ is local
with maximal ideal $\mathfrak{m} \subset A$
and in particular of finite Krull dimension. By induction on the Krull dimension, we may
assume that any regular ring of smaller Krull dimension (e.g., any further
localization of $A$) is $F$-smooth. 
If $A$ is zero-dimensional and hence a field, 
then we already know that $A$ is $F$-smooth: more generally, any regular ring in
characteristic $p$ is Cartier smooth and hence $F$-smooth. 
So suppose $\mathrm{dim}(A) > 0$. 
Choose $x \in \mathfrak{m} \setminus \mathfrak{m}^2$; then $A[1/x]$ and $A/x$
are $F$-smooth by 
induction on the dimension.
By 
\Cref{Fsmoothquotientregular}, it follows that $A$ is $F$-smooth. 
\end{proof}

For the proof that $F$-smoothness implies regularity, we will actually need much
less than $F$-smoothness itself. We expect that the result is related to 
recent works relating regularity to $p$-derivations \cite{HochsterJeffries,
Saito}. 

\begin{lemma} 
\label{regularitycrit}
Let $(A, \mathfrak{m}, k)$ be a complete intersection local ring. Then $A$ is
regular if and only if the map of $k$-vector spaces $H^{-1}( L_{A/\mathbb{Z}} \otimes_A k) \to H^{-1}(
L_{k/\mathbb{Z}})$ 
is injective. 
\end{lemma} 
\begin{proof} 
We have a transitivity triangle (for $\mathbb{Z} \to A \to k$), 
$L_{A/\mathbb{Z}} \otimes_A k  \to L_{k/\mathbb{Z}} \to L_{k/A}$, and
$L_{k/\mathbb{Z}}$ is concentrated in degrees $[-1, 0]$. 
Thus, the injectivity condition of the lemma is equivalent to the statement that
$H^{-2}(L_{k/A}) =0$, whence the result by \cite[Prop.~8.12]{Iyengar07}. 
\end{proof} 
\begin{proposition} 
\label{Fsmoothweakgivesregular}
Let $A$ be a 
complete intersection local noetherian ring with residue field $k$ of characteristic $p$. 
Then the following are equivalent: 
\begin{enumerate}
\item $A$ is regular.  
\item
$\mathrm{cofib}( \theta: A/p \to \mathcal{N}^p \Prism_A/p)
\otimes^{\mathbb{L}}_{A/p}
k \in \mathcal{D}^{\geq 1}(k)$ (e.g., this holds if $A$ is $F$-smooth by
\Cref{FSmoothWeak1} and its proof). 
\end{enumerate}
\end{proposition} 

We remind the reader that reduction mod $p$ is interpreted in the derived sense in this article, including in the statement above and the proof below.

\begin{proof} 
We have already shown above that regular rings are $F$-smooth, whence (1)
implies (2), so we show the converse. 
For any animated ring $B$,
the Nygaard fiber sequence of \cite[Rem.~5.5.8]{BhattLurieAPC} 
and the conjugate filtration on diffracted Hodge cohomology
\cite[Cons.~4.7.1]{BhattLurieAPC}
yields a fiber sequence in $\mathcal{D}(B/p)$,
\begin{equation}  \mathrm{cofib}( \theta: B/p \to \mathcal{N}^p \Prism_B/p) \to \bigwedge^p
L_{B/p/\mathbb{F}_p}  [-p] \to B/p
\label{fiberseqmodp} \end{equation}
In more detail, if $B$ is a polynomial $\mathbb{Z}$-algebra, then the 
Nygaard fiber sequence of \emph{loc.~cit.} gives a fiber sequence
\[ \mathcal{N}^p \Prism_{B} \to \mathrm{Fil}_p^{\mathrm{conj}} \OmegaD_B
\xrightarrow{\Theta + p} \mathrm{Fil}_{p-1} \OmegaD_B. 
\]
Using the eigenvalues of the action of the Sen operator $\Theta$ on the associated graded terms of
$\OmegaD_B$ (note that the conjugate filtration is just the Postnikov
filtration in this case) as in \cite[Notation~4.7.2]{BhattLurieAPC}, we find that 
$$H^*( \mathcal{N}^p \Prism_B) \simeq  \begin{cases} 
B/p & \ast =  1 \\
\widehat{\Omega}^p_{B/\mathbb{Z}} & \ast = p \\
0 & \text{otherwise}
 \end{cases} . 
$$
Moreover, the map $\theta: B/p \to \mathcal{N}^p \Prism_{B}/p$ is an isomorphism
in $H^0$ by comparison with the case $B = \mathbb{Z}$ (\Cref{HTofZp}); this
(together with left  Kan extension) easily gives the fiber sequence \eqref{fiberseqmodp}. 

Taking $B = A$ and base-changing to $A/ p \to k$, we obtain a fiber
sequence
\begin{equation} 
\label{fourthfibseq}
\mathrm{cofib}( \theta: A/p \to \mathcal{N}^p \Prism_A/p)
\otimes_{A/p}
k \to \bigwedge^p L_{A/\mathbb{Z}}[-p] \otimes_A k \to k
\end{equation} 
By the lci hypotheses, 
$L_{(A/p)/\mathbb{F}_p} \in \mathcal{D}(A/p)$ has
$\mathrm{Tor}$-amplitude in $[-1, 0]$. 
The condition (2) is equivalent to the injectivity of the map
(obtained by applying $H^0$ to the second map in \eqref{fourthfibseq}, using
d\'ecalage \cite[\S 4.3.2]{Illusiecotangent1})
\begin{equation} \label{naturaldividedpowmap} \Gamma^p H^{-1} ( L_{A/\mathbb{Z}} \otimes_A k) = H^{-p}( \bigwedge^p
L_{A/\mathbb{Z}} \otimes_A k) \to k .\end{equation}
We have constructed the map \eqref{naturaldividedpowmap} naturally in the lci
ring $(A, \mathfrak{m})$ with residue field $k$. 
Moreover, it is injective if $A = k$, since we have seen that regular rings
are $F$-smooth and hence satisfy (2). Conversely, suppose $A$ satisfies (2). It follows by naturality of
\eqref{naturaldividedpowmap} that $H^{-1} ( L_{A/\mathbb{Z}} \otimes_A k) \to
H^{-1}(L_{k/\mathbb{Z}})$ is injective, whence regularity of $A$ by 
\Cref{regularitycrit}. 
\end{proof} 

\begin{proof}[Proof that $F$-smoothness implies regularity under
$p$-completeness] 
Let $A$ be a $p$-complete noetherian ring which is $F$-smooth. We argue that $A$
is regular. It suffices to check that the localization of $A$ at any maximal
ideal is regular, since a noetherian ring is regular if and only if its
localizations at maximal ideals are regular. Since $p$ belongs to any maximal ideal, we reduce to the case
where $A$ is  a $p$-complete local ring which is $F$-smooth. 
Our $p$-quasisyntomicity assumption implies that $L_{A/\mathbb{Z}} \otimes_A k \in \mathcal{D}^{[-1,
0]}(k)$; by \cite[Prop.~1.8]{Avramov}, this implies that $A$ is a complete
intersection. 
Then we can appeal to \Cref{Fsmoothweakgivesregular} to conclude that $A$ is
regular, as desired. 
\end{proof} 

\subsection{Dimension bounds}
As an application, we can obtain some dimension bounds on the Hodge--Tate cohomology
of regular rings, and verify \cite[Conj.~10.1]{BhattLurieprismatization} with
an additional assumption of 
$F$-finiteness. Let us recall the setup. 
For a quasisyntomic ring $R$, we consider the Hodge--Tate stack
$\wcartHT_{\mathrm{Spf}(R)}$ 
defined in \cite[Cons.~3.7]{BhattLurieprismatization}; recall that this stack comes with
a map $\wcartHT_{\mathrm{Spf}(R)} \to \mathrm{Spf}(R)$ and 
line bundles
$\mathcal{O}_{\wcartHT_{\mathrm{Spf}(R)}}\left\{i\right\}$ whose cohomology yields
$\Prismbar_R\left\{i\right\}$.

Before formulating the result, let us also recall some facts about 
$F$-finiteness. 
A noetherian $\mathbb{F}_p$-algebra $S$ is said to be \emph{$F$-finite} if it is
finitely generated over its $p$th powers. 
If $S$ is  a noetherian local $\mathbb{F}_p$-algebra, $F$-finiteness is equivalent to
the 
assumption 
that the residue field of $S$ is $F$-finite and $S$ is excellent, cf.~\cite[Cor.~2.6]{Kunz}. 
Moreover, $S$ is $F$-finite if and only if the cotangent complex $L_{S/\mathbb{F}_p} \in
\mathcal{D}(S)$ is almost
perfect, 
cf.~\cite[Th.~3.6]{DundasMorrow} and \cite[Th.~3.5.1]{Orientations}.

\begin{corollary} Let $R$ be a $p$-complete regular local ring with residue field $k$. Suppose that 
$R/pR$ is $F$-finite. 
Let $ d= \dim R + \mathrm{log}_p[k: k^p]$. 
Then the Hodge--Tate stack $\wcartHT_{\mathrm{Spf}(R)}$ 
has
cohomological dimension $\leq d$. 
In particular, $\Prismbar_{R}\left\{i\right\} \in \mathcal{D}^{\leq d}(\mathbb{Z}_p)$ for
each $i$.
\end{corollary}

\begin{proof} 
 Let us first
reduce to the case where $R$ is complete. 
The map $\hat{R}_{\mathfrak{m}} \otimes_{\mathbb{Z}} L_{R/\mathbb{Z}} \to L_{\hat{R}_{\mathfrak{m}}/\mathbb{Z}}$
is an isomorphism after $p$-completion: in fact, both sides are almost perfect
mod $p$ by $F$-finiteness (as recalled above) 
 and the map is an isomorphism after base-change to the residue field, whence
 the claim by Nakayama. 
It follows by \cite[Rem.~3.9]{BhattLurieprismatization} (and its proof) that 
the diagram
\[ \xymatrix{
\wcartHT_{\mathrm{Spf}(\hat{R}_{\mathfrak{m}})} \ar[d]  \ar[r] &
\wcartHT_{\mathrm{Spf}(R)} \ar[d]  \\
\mathrm{Spf}(\hat{R}_{\mathfrak{m}}) \ar[r] &  \mathrm{Spf}(R)
},\]
is cartesian. Therefore, since $R \to \widehat{R}_{\mathfrak{m}}$ is faithfully
flat, 
it suffices 
to replace everywhere $R$ by $\hat{R}_{\mathfrak{m}}$, so  we may assume that $R$
itself is complete.

Let us now verify the dimension bound on the Hodge--Tate complexes
$\Prismbar_R\left\{i\right\}$, i.e., that $\Prismbar_R\left\{i\right\} \in
\mathcal{D}^{\leq d}( \mathbb{Z}_p)$ for each $i$. 
The associated graded terms of the Nygaard filtration on
$\Prismbar_R\left\{i\right\}$ (i.e., $\mathcal{N}^j \Prism_R/p^{(-1)}$) are almost perfect
$R$-modules, whence $\mathfrak{m}$-adically complete, 
in light of the Nygaard fiber sequences \cite[Rem.~5.5.8]{BhattLurieAPC} and the
almost perfectness mod $p$ of  $L_{R/\mathbb{Z}}$ recalled above. 
 Using the 
completeness of the Nygaard filtration (\Cref{regularringsSegal}), we find that
$\Prismbar_R\left\{i\right\}$ is
$\mathfrak{m}$-adically complete.  
If $R$ is zero-dimensional and hence $R = k$, the result follows from the
comparison \cite[Th.~5.4.2]{BhattLurieAPC} between Hodge--Tate and de Rham cohomology 
of $\mathbb{F}_p$-algebras
since
$\dim \Omega^1_{k/\mathbb{F}_p} = \mathrm{log}_p[k: k^p]$, cf.~\cite[Tag
07P2]{stacks-project}. 
Otherwise, choose $x \in \mathfrak{m} \setminus \mathfrak{m}^2$. 
The ring $R/x$ is also regular local with the same residue field and of
dimension one less. 
To see that $\Prismbar_R\left\{i\right\} \in \mathcal{D}^{\leq d}(R)$, it suffices (by
$x$-adic completeness proved above) to
show that $\Prismbar_R\left\{i\right\} /x \in \mathcal{D}^{\leq d}(R)$. However, we have a
fiber sequence
from \Cref{prismHTfibseq} which, together with induction on the dimension, implies the claim.

Now we prove the cohomological dimension bound on $\wcartHT_{\mathrm{Spf}(R)}$. 
First, we prove that the cohomological dimension is at most $d+1$. 
Let $W$ be a Cohen ring for $k$. By the Cohen structure theorem, we have a surjection
\[ A = W[[t_1, \dots, t_r]] \to R  \]
for $r = \mathrm{dim}(R)$, whose kernel is generated by a nonzerodivisor. 
By choosing a $p$-basis for $k$, we see that the ring $A$ is formally \'etale over a polynomial ring in $d$ variables over
$\mathbb{Z}_p$, and consequently that $\wcartHT_{\mathrm{Spf}(A)} =
\wcartHT_{\spf( \mathbb{Z}_p\left \langle x_1, \dots, x_d
\right\rangle)
} \times_{\spf( \mathbb{Z}_p\left \langle x_1, \dots, x_d
\right\rangle)}
\spf(A)$. Using the expression for the Hodge--Tate stack of the polynomial
$\mathbb{Z}_p$-algebra in \cite[Ex.~9.1]{BhattLurieprismatization} as the
classifying stack of $(\mathbb{G}_a^{\sharp d} \rtimes \mathbb{G}_m^{\sharp})$, 
and the explicit description of representations of $\mathbb{G}_a^{\sharp},
\mathbb{G}_m^{\sharp}$ in \cite[Sec.~3.5]{BhattLurieAPC} and
\cite[Lem.~6.7]{BhattLurieprismatization}, one finds that 
$\mathrm{cd}( \wcartHT_{\mathrm{Spf}(A)}) \leq d+1$. 
By affineness of $\wcartHT_{\mathrm{Spf}(R)} \to \wcartHT_{\mathrm{Spf}(A)}$
(\Cref{affinelemma} below), we
obtain 
$\mathrm{cd}(\wcartHT_{\mathrm{Spf}(R)} ) \leq d+1$. It remains to show that
$H^{d+1}$ of any quasicoherent sheaf on  
$\wcartHT_{\mathrm{Spf}(R)}$ (which we may assume to be $p$-torsion) vanishes.

Consider the category of $p$-torsion sheaves on 
$\wcartHT_{\mathrm{Spf}(R)}$ (we recall that $\wcartHT_{\mathrm{Spf}(R)}$ is
defined as a functor on $p$-nilpotent rings, so this case will suffice).  
We claim that 
for any $p$-torsion sheaf $\mathcal{F}$ on  
$\wcartHT_{\mathrm{Spf}(R)}$, we have $H^0( \wcartHT_{\mathrm{Spf}(R)}, 
\mathcal{F}\left\{-n\right\}) \neq 0$ for some $n$. 
In fact, using the affine map $ \wcartHT_{\mathrm{Spf}(R)} \to 
\wcartHT_{\mathrm{Spf}(A)}
$, this claim reduces to the analogous claim for $\mathrm{Spf}(A)$; but this in
turn follows from the explicit description of the Hodge--Tate stack for
$\mathrm{Spf}(A)$. 
It follows that the category of $p$-torsion sheaves on 
$\wcartHT_{\mathrm{Spf}(R)}$ is generated under 
colimits and extensions by the quotients of the $\mathcal{O}\left\{n\right\}/p$. 
Since $H^{d+1}$ is a right exact functor for quasicoherent sheaves on
$\wcartHT_{\mathrm{Spf}(R)}$ (as proved above), and since we have
the cohomological dimension bound on the $\mathcal{O}\left\{n\right\}$, we now
conclude as desired. 
\end{proof}

\begin{lemma} 
\label{affinelemma}
Let $R$ be a quasisyntomic ring and let $t \in R$ be a nonzerodivisor. The map 
$ \wcartHT_{\spf(R/t)} \to \spf(R/t) \times_{\spf(R)} \wcartHT_{\spf(R)}$ is affine. 
\end{lemma} 
\begin{proof} 
We reduce to the case where $R $ is the $p$-completion of $\mathbb{Z}_p[t]$. 
In this case, by 
\cite[Ex.~9.1]{BhattLurieprismatization},
the above map is identified with 
$B\mathbb{G}_m^{\sharp} \to B (\mathbb{G}_a^{\sharp} \rtimes
\mathbb{G}_m^{\sharp})$; in particular, it is affine. 
\end{proof} 

\section{Comparison with $p$-adic \'etale Tate twists}
\label{CompEtaleTate}
In this section, we prove \Cref{Zpiregularschemes} from the introduction. That
is, on a $F$-smooth $p$-torsionfree
scheme, we show that 
the complex $\mathbb{Z}/p^n(i)_X$ can be obtained via a generalization of the
construction of $p$-adic  \'etale Tate twists \cite{Gei04, Schneider, Sato},
i.e., by modifying the truncated $p$-adic vanishing cycles $\tau^{\leq i} Rj_* (
\mu_{p^n}^{\otimes i})$ by taking the subsheaf in degree $i$ generated by
symbols from $X$. 

Our strategy is as follows. Since we already know the $\mathbb{F}_p(i)$ for
$\mathbb{Z}[1/p]$-schemes are the usual Tate twists, it suffices to treat the
$p$-henselian case. One needs to show that the map 
$\mathbb{F}_p(i)(X) \to \mathbb{F}_p(i)(X[1/p])$ is highly coconnected. 
The \'etale comparison theorem (\Cref{etcompthm} below) implies that one may
obtain the $\mathbb{F}_p(i)(X[1/p])$ by inverting the operator $v_1$ on the
$\mathbb{F}_p(i)(X)$.  Thus, we reduce to showing that the map 
$v_1: \mathbb{F}_p(i)(X) \to \mathbb{F}_p(i+p-1)(X)$ is highly coconnected. 
Here we use an explicit argument (which was inspired by \cite{HW21}) with the
expression of \cite{BMS2} to check the claim. To determine the top-degree
cohomology, we use also the classical results of Bloch--Kato \cite{BK86} on
$p$-adic vanishing cycles. 

In \cite{KM21, KST21}, it is shown that the 
description of the $\mathbb{Z}/p^n(i)_X$ for regular $\mathbb{F}_p$-schemes via
logarithmic Hodge--Witt forms (cf.~\cite[Sec.~8]{BMS2} and \cite{GH99}) also
holds for the Cartier smooth case. 
Our \Cref{Zpiregularschemes} may be seen as a mixed characteristic analog
of this result.

\subsection{The \'etale comparison}

Let $X$ be a qcqs derived scheme. 
As in \cite[Sec.~8]{BhattLurieAPC}, one associates the graded $E_\infty$-algebra
$\bigoplus_{i \in \mathbb{Z}} \mathbb{F}_p(i)(X)$, the mod $p$ syntomic
cohomology of $X$. 
When $X$ is the spectrum of a $p$-complete animated ring, this can be obtained
(via descent and left Kan extension) from the Frobenius fixed points 
of prismatic cohomology  as in \cite[Sec.~7]{BhattLurieAPC} and \cite{BMS2}. However, when $X $ is a $\mathbb{Z}[1/p]$-scheme, it
is the usual Tate twisted \'etale cohomology
$\bigoplus_{i \in \mathbb{Z} }R \Gamma_{\et}( X, \mu_p^{\otimes i})$. 

For any $X$, 
the class $v_1 \in H^0( \mathbb{F}_p(p-1)(\mathbb{Z}))$ constructed in
\Cref{constructionofv1} yields a class in $H^0
\left( \bigoplus_i 
\mathbb{F}_p(i)(X)\right)$ which maps to a unit after passage to $X[1/p]$. 

\begin{theorem}[{The \'etale comparison, \cite[Th.~8.5.1]{BhattLurieAPC}}] 
\label{etcompthm}
Let $X$ be any qcqs derived scheme. 
The natural map of graded $E_\infty$-algebras over $\mathbb{F}_p$,
\begin{equation}  \bigoplus_{i  \in \mathbb{Z}} \mathbb{F}_p(i)(X) \to \bigoplus_{i \in \mathbb{Z}}
\mathbb{F}_p(i)(X[1/p]),  \label{mapofgradedrings} \end{equation}
exhibits the target as the localization of the source at $v_1$. 
In particular, 
for any $i$, the filtered colimit
\[ \mathbb{F}_p(i)(X) \stackrel{v_1}{\to} \mathbb{F}_p(i+p-1)(X)
\stackrel{v_1}{\to} \mathbb{F}_p(i + 2(p-1))(X) \to \dots    \]
is canonically identified with $\mathbb{F}_p(i)(X[1/p]) = R \Gamma_{\et}( X[1/p]; \mathbb{F}_p(i))$. 
\end{theorem} 
\begin{proof} 
When $X$ is a scheme over $\mathbb{Z}[\zeta_{p^\infty}]$, the result is proved
in \cite[Th.~8.5.1]{BhattLurieAPC}: in that case, one obtains a similar
statement for the $p$-complete $E_\infty$-algebras 
$\bigoplus_{i \in \mathbb{Z}} \mathbb{Z}_p(i)(X), \bigoplus_{i \in \mathbb{Z}}
\mathbb{Z}_p(i)(X[1/p])$, when one inverts the class $\epsilon \in H^0(
\mathbb{Z}_p(1)( \mathbb{Z}[\zeta_{p^\infty}]))$ arising from the given system
of $p$-power roots of unity. Let us explain how one can deduce the current form
of the result. 

First, if $X$ is $p$-quasisyntomic (which is the only case that will be used
below), then we observe that both sides  
of \eqref{mapofgradedrings} are coconnective. Using the sheaf property, one may
reduce to the case where $X$ lives over $\mathbb{Z}[\zeta_{p^\infty}]$, which is
proved in \emph{loc.~cit.}

To prove the result more generally, it suffices to show that the construction
which carries an animated ring $R$ to $R \Gamma_{\et}( \spec(R[1/p]),  \mu_p^{\otimes
i})$ (i.e., the right-hand-side of \eqref{mapofgradedrings}) 
is left Kan extended from smooth $\mathbb{Z}$-algebras. 
In fact, the left-hand-side is left Kan extended from smooth
$\mathbb{Z}$-algebras \cite[Prop.~8.4.10]{BhattLurieAPC}, as is its localization 
after inverting $v_1$, and for smooth (in particular $p$-quasisyntomic) algebras
we have already seen the result.

Now we claim that the construction which carries
an animated ring $R$ to $\bigoplus_{i \in \mathbb{Z}}R \Gamma_{\et}( \spec(R \otimes
\mathbb{Z}[\zeta_{p^\infty}][1/p]), \mu_p^{\otimes i})$ is left Kan extended
from smooth $\mathbb{Z}$-algebras. 
In fact, by \cite[Th.~8.5.1]{BhattLurieAPC}, this construction is the
localization of $R \mapsto \bigoplus_{i \in \mathbb{Z}} \mathbb{F}_p(i)( \spec(R
\otimes_{\mathbb{Z}} \mathbb{Z}[\zeta_{p^\infty}]))$ at $v_1$.
This construction in turn fits into a fiber sequence
\cite[Rem.~8.4.8]{BhattLurieAPC} involving terms that are either rigid for
henselian pairs or which commute with sifted colimits, cf.~the proof of
\cite[Prop.~8.4.10]{BhattLurieAPC}.  As in \emph{loc.~cit.}, this implies that 
$\bigoplus_{i \in \mathbb{Z}}R \Gamma_{\et}( \spec(R \otimes
\mathbb{Z}[\zeta_{p^\infty}][1/p]), \mu_p^{\otimes i})$
is left Kan extended from smooth $\mathbb{Z}$-algebras. Taking
$\mathbb{Z}_p^{\times}$-Galois invariants, we conclude that 
$\bigoplus_{i \in \mathbb{Z} }R \Gamma_{\et}( \spec(R[1/p]),  \mu_p^{\otimes
i})$
has the desired left Kan extension property. 
\end{proof}

\subsection{Comparison with the generic fiber}

In this subsection, we prove the following basic comparison result; over a
perfectoid, this has also been proved by Bouis, cf.~\cite[Th.~4.14]{Bouis}. 
\begin{proposition} 
Let $A$ be a $p$-torsionfree $p$-quasisyntomic ring which is $F$-smooth. Then for each $i$,
the canonical map 
$\mathbb{F}_p(i)(A) \to \mathbb{F}_p(i)(A[1/p]) = R \Gamma_{\et}(\spec(A[1/p]) ; \mathbb{F}_p(i))$  has fiber in
$\mathcal{D}^{\geq i+1}(\mathbb{F}_p)$. 
\label{mainresult}
\end{proposition} 

Without loss of generality, we may assume $A$ is $p$-henselian. 
To prove this result, we use \Cref{etcompthm}. Using this, we are reduced to understanding the effect of multiplying with the class $v_1$  on the syntomic cohomology of $A$. Recall that the latter is defined as an equalizer:
\[ \mathbb{F}_p(i)(A) =  \mathrm{eq} \left( \mathcal{N}^{\geq p-1} \Prismp{A}{p-1} \rightrightarrows \Prismp{A}{p-1}\right).\]
of the Frobenius and canonical maps. To analyze the behaviour of cupping with $v_1$ with respect to the fibre of the canonical map above, we shall use the relation of $v_1$ with $\tilde{\theta}$ and the following result:

\begin{lemma} 
\label{lemfiltsegal}
Let $A$ be a $p$-torsionfree $p$-quasisyntomic ring which is $F$-smooth. 
Then for each $i, j$, the fiber of the multiplication map
\begin{equation} \label{filteredv1} \tilde{\theta}: \mathcal{N}^{\geq j} \Prismp{A}{i} \to \mathcal{N}^{\geq
j+p} \Prismp{A}{i+p-1}  \end{equation}
belongs to $\mathcal{D}^{\geq j+2}(\mathbb{F}_p)$. 
\end{lemma} 
\begin{proof} 
The $F$-smoothness assumption shows that for each $j'$, the fiber of 
$\theta: \mathcal{N}^{j'} \frac{\Prism_A}{p} \to \mathcal{N}^{j' + p}
\frac{\Prism_A}{p}$ belongs to $\mathcal{D}^{\geq j'+2}(\mathbb{F}_p)$. By filtering both
sides (by the Nygaard filtration, which is complete by $F$-smoothness) of \eqref{filteredv1}, the conclusion of the
lemma follows, in light of \Cref{widetildethetaimage}. 
\end{proof} 

\begin{proposition} 
Suppose $A$ is a $p$-torsionfree $p$-quasisyntomic ring which is $F$-smooth.
For each $i \in \mathbb{Z}$, the Frobenius map
\begin{equation} \label{phiimap}  \phi_i:  
\mathcal{N}^{\geq i} \Prismp{A}{i} \to \Prismp{A}{i} 
\end{equation}
has fiber in $\mathcal{D}^{\geq i+2}(\mathbb{F}_p)$. 
\end{proposition}

\begin{proof}
In fact, this follows because 
the map \eqref{phiimap} 
admits a complete descending filtration, indexed over $j \geq i$, with 
$\mathrm{gr}^j$ given by 
$\phi_j: \mathcal{N}^j \frac{\Prism_A}{p} \to
\frac{{\Prismbar_A} \left\{j\right\}}{p}$; this is clear from the definition of
the Nygaard filtration via descent from quasiregular semiperfectoid rings; now
$F$-smoothness gives the cohomological bound on the fiber of the map on
associated graded terms. 
\end{proof}

\begin{proof}[Proof of \Cref{mainresult}] 

We will show that the map
\begin{equation} \label{v1fib}  v_1: \mathbb{F}_p(i)(A) \to
\mathbb{F}_p(i+p-1)(A)  \end{equation}
has fiber in $\mathcal{D}^{\geq i+1}(\mathbb{F}_p)$; this will suffice thanks to the
\'etale comparison (\Cref{etcompthm}). 
Without loss of generality, we can assume $A$ is $p$-henselian. 
By construction, the fiber of \eqref{v1fib} is the equalizer of the two maps
(arising from the canonical map and divided Frobenius map)
\begin{equation} \label{twokeymaps}  \mathrm{fib}\left( \mathcal{N}^{\geq i} \Prismp{A}{i} 
\xrightarrow{v_1} \mathcal{N}^{\geq
i+p-1}\Prismp{A}{i+p-1} \right)
\rightrightarrows  
\mathrm{fib}\left( \Prismp{A}{i}  \xrightarrow{v_1} \Prismp{A}{i+p-1} \right)
.\end{equation}

By assumption, since $A$ is $F$-smooth, the Frobenius maps 
$$\phi_i : \mathcal{N}^{\geq i} \Prismp{A}{i} \to \Prismp{A}{i}, \quad 
\phi_{i+p-1} : 
\mathcal{N}^{\geq i+p-1} \Prismp{A}{i+p-1} \to \Prismp{A}{i+p-1} $$
have fibers in $\mathcal{D}^{\geq i+2}(\mathbb{F}_p)$. 
Therefore, by taking fibers of multiplication by $v_1$, we find that the fiber
of the Frobenius maps in \eqref{twokeymaps} belong to $\mathcal{D}^{\geq
i+2}(\mathbb{F}_p)$. 

Now consider the canonical map
in \eqref{twokeymaps}; we claim that 
it induces the zero map in cohomological degrees
$\leq i$. 
To see this, 
we observe that the canonical map factors through the map 
\begin{equation}  \mathrm{fib}\left( \mathcal{N}^{\geq i} \Prismp{A}{i} 
\xrightarrow{v_1} \mathcal{N}^{\geq
i+p-1}\Prismp{A}{i+p-1} \right)
\to 
\mathrm{fib}\left( \mathcal{N}^{\geq i-1} \Prismp{A}{i} 
\xrightarrow{\tilde{\theta}} \mathcal{N}^{\geq
i+p-1}\Prismp{A}{i+p-1} \right)
\end{equation}
as $v_1 \in \mathcal{N}^{\geq p-1} \frac{\Prism_{\mathbf{Z}_p}\{p-1\}}{p}$ lifts
to $\tilde{\theta} \in \mathcal{N}^{\geq p}
\frac{\Prism_{\mathbf{Z}_p}\{p-1\}}{p}$. However, we have seen that the
right-hand-side of the above belongs to $ \mathcal{D}^{\geq
i+1}(\mathbb{F}_p)$ thanks to \Cref{lemfiltsegal}. 
This implies that the canonical map vanishes in degrees $\leq i$. 

Thus, we find that 
the desired fiber of the map \eqref{v1fib}
is the equalizer of two maps \eqref{twokeymaps}, one of which has fiber in
$\mathcal{D}^{\geq i+2}(\mathbb{F}_p)$, and one of which is zero in degrees $\leq i$. 
This implies the result. 
\end{proof} 

\subsection{Generation by symbols}

In this section, we complete the proof of \Cref{Zpiregularschemes} from the
introduction. 
First, we prove the following basic symbolic generation result. 
For more refined results about the connection of the  
$\{H^i(
\mathbb{Z}/p^n(i)(R))\}$ to $p$-adic Milnor $K$-theory, cf.~\cite{LM21}. 
In the following, we use that for any ring $R$, we have a natural
Kummer map $R^{\times} \to H^1( \mathbb{Z}_p(1)(R))$, cf.~\Cref{Zp1Gm}. 
Iterating, we obtain a ``symbol'' map $(R^{\times})^{\otimes i} \to H^i(
\mathbb{Z}_p(i)(R))$. 
\begin{proposition} 
\label{symbolicgen}
For any strictly henselian local ring $R $, the symbol map
$(R^{\times})^{\otimes i} \to H^i(
\mathbb{Z}/p^n(i)(R))$ is surjective. 
\end{proposition} 

To prove \Cref{symbolicgen}, it clearly suffices to assume that $R$ is $p$-henselian and that $n = 1$, using
the connectivity bound 
$\mathbb{Z}/p^n(i)(R) \in \mathcal{D}^{\leq i}(\mathbb{Z}/p^n)$, cf.~\cite[Cor.~5.43]{AMNN}. 
By the left Kan extension property of the $\mathbb{F}_p(i)(-)$ for $p$-henselian
rings (\cite[Th.~5.1]{AMNN} or \cite[Prop.~7.4.8]{BhattLurieAPC}), we may assume that $R$ is the strict henselization at a
characteristic $p$ point of a smooth $\mathbb{Z}$-scheme. 
In this case, we know by \Cref{mainresult} and \Cref{regularringsSegal} that 
the natural map induces an injection
\begin{equation}  H^i( \mathbb{F}_p(i)(R)) \subset H^i( \spec(R[1/p]), \mu_p^{\otimes i}) =
H^i( \mathbb{F}_p(i)(R[1/p])),  \label{injection}  \end{equation}
and we will identify the left-hand-side as the subgroup of the right-hand-side
generated by symbols. 

We now recall some of the work of Bloch--Kato \cite{BK86}, which 
describes the right-hand-side of \eqref{injection}; it will be convenient to
formulate the assertion sheaf-theoretically. 

Let $X $ be a smooth $\mathbb{Z}$-scheme, let $j: X[1/p] \subset X, i: Y
\stackrel{\mathrm{def}}{=} X
\times_{\spec(\mathbb{Z})}
\spec(\mathbb{F}_p) \subset X$ be the respective open and closed immersions
corresponding to the ideal $(p)$. 
The work \cite{BK86} 
describes the \'etale sheaves of $\mathbb{F}_p$-modules on $Y$,
\begin{equation} M^i \stackrel{\mathrm{def}}{=} i^* R^ij_* (\mu_{p}^{\otimes i}) .  \end{equation}
In particular, 
using the map $i^* j_* \mathcal{O}_X^{\times} \to M^1$ arising from the Kummer
sequence and the graded ring structure on the $\left\{M^i\right\}$, 
one has a symbol map
\begin{equation}  \label{genericsymbolmap} i^* j_*(
\mathcal{O}_{X[1/p]}^{\times})^{\otimes i} \to M^i.
\end{equation}
By \cite[Th.~14.1]{BK86}, the symbol map is surjective. 
Moreover, by \cite[6.6]{BK86}, one has a surjective residue map of
$\mathbb{F}_p$-sheaves
\begin{equation} \label{resmap} \mathrm{res}: M^i \twoheadrightarrow \Omega^{i-1}_{Y,
\mathrm{log}} . \end{equation}

\begin{proposition} 
\label{symbolsworkedout}
If $R$ is a strictly henselian local ring which is ind-smooth over $\mathbb{Z}$,
then
the kernel of the surjective residue map 
\eqref{resmap}
$H^i( \spec(R[1/p]), \mu_p^{\otimes i}) \to \Omega^{i-1}_{R/p, \mathrm{log}}$
is the subgroup of $H^i( \spec(R[1/p]), \mu_p^{\otimes i})$ generated by symbols from $R$, i.e., by
the image of $R^{\times} \otimes \dots \otimes R^{\times}$ under the symbol map
$R[1/p]^{\times} \otimes \dots R[1/p]^{\times} \twoheadrightarrow 
H^i( \spec(R[1/p]), \mu_p^{\otimes i})$, as in  
\eqref{genericsymbolmap}. 
\end{proposition} 
\begin{proof} 
Let $B \subset H^i( \spec(R[1/p]), \mu_p^{\otimes i})$ be the subgroup generated
by the symbols from $R$. 
The Bloch--Kato filtration \cite[Cor.~1.4.1]{BK86} gives a short exact sequence
\[ 0 \to \Omega^{i-1}_{R/p} \to 
H^i( \spec(R[1/p]), \mu_p^{\otimes i}) \to \Omega^{i}_{R/p, \mathrm{log}} \oplus
\Omega^{i-1}_{R/p, \mathrm{log}} \to 0,
\]
where the second map $H^i( \spec(R[1/p]), \mu_p^{\otimes i}) \to
\Omega^{i-1}_{R/p, \mathrm{log}}$ is the residue \eqref{resmap}. 
By construction of the filtration and the first map \cite[4.3]{BK86}, one sees
that $B$ contains the subgroup 
$\Omega^{i-1}_{R/p} $. 
As in \cite[6.6]{BK86}, 
the map 
$H^i( \spec(R[1/p]), \mu_p^{\otimes i}) \to \Omega^{i}_{R/p, \mathrm{log}} \oplus
\Omega^{i-1}_{R/p, \mathrm{log}}$ 
carries the symbol 
$r_1 \otimes \dots \otimes r_i$ for 
$r_1, \dots, r_i \in R^{\times}$ to 
$( \frac{dr_1}{r_1} \wedge \dots \wedge \frac{dr_i}{r_i}, 0)$, and the symbol 
$r_1 \otimes \dots \otimes r_{i-1} \otimes p$ to 
$(0, 
\frac{dr_1}{r_1} \wedge \dots \wedge \frac{dr_{i-1}}{r_{i-1}})$. 
From this, 
one
sees that $H^i( \spec(R[1/p]), \mu_p^{\otimes i})/B \xrightarrow{\sim}
\Omega^{i-1}_{R/p, \mathrm{log}}$ via the residue, as claimed. 
\end{proof}

Now we return to the proof of \Cref{symbolicgen}, and identify the image of
\eqref{injection}. 
The $\mathcal{D}(\mathbb{F}_p)$-valued sheaf $\mathbb{F}_p(i)(-)$ 
restricts to an object (with the same notation) on the category of ind-smooth, $p$-henselian $\mathbb{Z}$-algebras $R$. 
For any such $R$, we have natural maps
from \eqref{injection} and \eqref{resmap}, 
\[ \mathbb{F}_p(i)(R) \to \mathbb{F}_p(i)(R[1/p]) \xrightarrow{\mathrm{res}} 
R \Gamma_{\et}( \spec(R/p), \Omega^{i-1}_{\cdot, \mathrm{log}})[-i] =
\mathbb{F}_p(i-1)(R/p)[-1],
\]
where the last identification is \cite[Sec.~8]{BMS2} (and reviewed in
\Cref{syntomicequalchar}).
We claim that the composite vanishes. 
In fact, this is true for any such map. 
\begin{proposition} 
\label{naturalcompositevanishes}
Any natural map 
$\mathbb{F}_p(i)(R) \to \mathbb{F}_p(i-1)(R/p)[-1]$, 
 defined on $p$-henselian ind-smooth $\mathbb{Z}$-algebras $R$,
vanishes. 
\end{proposition} 
\begin{proof} 
By left Kan extension, we can define a natural map on all quasisyntomic
$\mathbb{Z}_p$-algebras $R$,
$\mathbb{F}_p(i)(R) \to \mathbb{F}_p(i-1)(R/p)[-1]$. 
Both sides define $\mathcal{D}(\mathbb{F}_p)$-valued sheaves for the quasisyntomic
topology. The source is discrete as a sheaf (by the odd vanishing theorem,
\cite[Th.~4.1]{Prisms})
and the target is concentrated in cohomological degree $1$, whence the map must
vanish. 
\end{proof} 

\begin{proof}[Proof of \Cref{symbolicgen}] 
As before, we may assume that $R$ is ind-smooth over $\mathbb{Z}$ and that $n
=1$. 
We have seen that the map $H^i(\mathbb{F}_p(i)(R)) \to
H^i(\mathbb{F}_p(i)(R[1/p]))$ is injective, and its image must contain the
image of $(R^{\times})^{\otimes i}$. 
The image of $H^i( \mathbb{F}_p(i)(R))$ is contained in the kernel
of the residue map thanks to \Cref{naturalcompositevanishes}. 
But by \Cref{symbolsworkedout}, the kernel of the residue maps on $H^i (
\mathbb{F}_p(i)(R[1/p]))$ is precisely the image of $(R^{\times})^{\otimes i}$. The result follows. 
\end{proof} 

\begin{proof}[Proof of \Cref{Zpiregularschemes}] 
Let $X$ be a $p$-torsionfree scheme which is $F$-smooth. 
Thanks to \Cref{mainresult}, the map 
$\mathbb{Z}/p^n(i)_X \to Rj_* ( \mu_{p^n}^{\otimes i})$ has homotopy fiber in 
degrees $\geq i+1$. Since $\mathbb{Z}/p^n(i)_X$ is concentrated in degrees $[0,
i]$ by \cite[Cor.~5.43]{AMNN}, it suffices to identify the image of the
(injective) map $\mathcal{H}^i( \mathbb{Z}/p^n(i)_X) \to  R^i j_* (
\mu_{p^n}^{\otimes i})$. The claim is that it is exactly the subsheaf generated
by symbols on $X$. This follows thanks to the symbolic generation of the source
(\Cref{symbolicgen}). 
\end{proof} 

\subsection{Comparison with Geisser--Sato--Schneider}

In this section, we use the above results to compare the $\mathbb{Z}/p^n(i)_X$
with the complexes defined by Sato \cite{Sato} for semistable schemes, cf.~also earlier work of
Schneider  \cite{Schneider} and Geisser \cite{Gei04} for the smooth case; such a
comparison 
was predicted in \cite[Rem.~1.16]{BMS2}. 

Let $X$ be a regular scheme of finite type over a Dedekind domain $A$ such that
every characteristic $p$ residue field of $A$ is perfect. 
Suppose that $X$ is semistable over characteristic $p$ points of
$\spec(A)$. 
For $n, i \geq 0$, Sato \cite{Sato} constructs objects $\mathfrak{I}_n(i)_X \in
\mathcal{D}^{[0, i]}(X_{\et}, \mathbb{Z}/p^n \mathbb{Z})$ and conjectures
\cite[Conjecture~1.4.1]{Sato} that they can be identified with the \'etale
sheafification of the motivic (cycle) complexes mod $p^n$; in the smooth case this
follows from \cite{Gei04}.  
Here we compare the $\mathfrak{I}_n(i)_X$ to the $\mathbb{Z}/p^n(i)_X$. 

\begin{theorem} 
There is a canonical, multiplicative equivalence 
$\mathfrak{I}_n(i)_X \simeq \mathbb{Z}/p^n (i)_X$ of objects in
$\mathcal{D}^b(X_{\et},
\mathbb{Z}/p^n)$. 
\end{theorem}

\begin{proof} 
As in \cite[\S 4.2]{Sato}, 
the complex $\mathfrak{I}_n(i)_X$ is built as the mapping fiber of a map 
from 
$\tau^{\leq i}Rj_*( \mu_{p^n}^{\otimes i}) $ to the $(-i)$-suspension of a
discrete sheaf. 
Therefore, in order to verify the comparison, it suffices (by combining
\Cref{mainresult}, \Cref{regularringsSegal}, and \Cref{symbolicgen}) to show
that the \'etale sheaf 
$\mathcal{H}^i( \mathfrak{I}_n(i)_X)$
is generated by symbols. 
We may assume $n = 1$ for this and work stalkwise. 

Let $R$ denote the strict henselization of a characteristic $p$ point $x \in X$. 
We can replace $A$ by its strict henselization, which is a mixed characteristic
DVR; let $\pi \in A$ denote the uniformizer. 
Consider  the $\mathbb{F}_p$-vector space
$H^i( \spec(R[1/p]), \mu_p^{\otimes i})$. 
We have a 
symbol map 
$\left(R[1/p]^{\times}\right)^{\otimes i} \to H^i( \spec(R[1/p]), \mu_p^{\otimes
i})$. 
Let $ F \subset H^i( \spec(R[1/p]), \mu_p^{\otimes
i})$ be the subgroup generated by the images of 
$(R^{\times})^{\otimes i}$ and $(1 + \pi R)^{\times} \otimes
(R[1/p]^{\times})^{\otimes i-1}$ under the symbol map, cf.~\cite[\S 3.4]{Sato}. 
As in \cite[Def.~4.2.4]{Sato}, the image of the injective map 
$\mathcal{H}^i( \mathfrak{I}_1(i)_X)_x \to H^i( \spec(R[1/p]), \mu_p^{\otimes i})$ is
exactly the subgroup $F$. 

Our observation is that the image of 
$(1 + \pi R)^{\times} \otimes
(R[1/p]^{\times})^{\otimes i-1}$ under the symbol map is actually contained in
the image of $(R^{\times})^{\otimes i}$. 
Since $R$ is a UFD (as a regular local ring), we have $R[1/p]^{\times} = \pi^{\mathbb{Z}} \oplus
R^{\times}$. 
Consider a symbol 
$(1 + \pi a) \otimes b_1 \otimes \dots \otimes b_{i-1}$ for $b_1, \dots,
b_{i-1} \in R[1/p]^{\times}$. 
Using the unique factorization, as well as the fact that $\pi \otimes (-\pi)$
maps to zero in $H^{2}( \spec(R[1/p]), \mu_p^{\otimes 2})$, 
we reduce to the case $i = 2$. 

Therefore, it suffices to show that, for $a \in R$, the image of 
$(1 + \pi a) \otimes \pi$ in $H^2( \spec(R[1/p]), \mu_p^{\otimes 2})$
belongs to the image of $R^{\times} \otimes R^{\times}$. 
By bilinearity, we may assume that $a \in R$ is a unit (e.g., if $a$ is not a
unit, we write 
$(1 + \pi a) = \frac{1  + \pi a }{1 + \pi (a + 1)} ( 1 + \pi(a+1))$). In this case, 
$(1 + \pi a) \otimes (-\pi a )$ maps to zero (cf.~\cite[Th.~3.1]{Tate76}).  
Using bilinearity again, it follows that $(1 + \pi a ) \otimes \pi$ maps to an
element of 
$H^2( \spec(R[1/p], \mu_p^{\otimes 2})$
in the image of $R^{\times} \otimes R^{\times}$. 

Consequently, it follows that the ring
$\bigoplus_{i \geq 0} \mathcal{H}^i(\mathfrak{I}_1(i)_X)_{x}$  is generated by symbols, whence we conclude. 
\end{proof}

\begin{example}
Let $K$ be a discretely valued field of mixed characteristic, and let
$\mathcal{O}_K \subset K$ be the ring of integers; let $k$ be the residue field. 
Let $X$ be a smooth scheme over $\mathcal{O}_K$ with special
fiber $k$.  Then the above results
(together with the description of $p$-adic nearby cycles in \cite{BK86},
cf.~\Cref{symbolsworkedout})
show that we have a natural cofiber sequence
in $\mathcal{D}(X_{\mathrm{et}}, \mathbb{Z}/p^n)$,
\begin{equation}  \mathbb{Z}/p^n(i)_X \to \tau^{\leq i}R j_*
(\mu_{p^n}^{\otimes i}) \to W_n \Omega^{i-1}_{X_k, \mathrm{log}}[-i],     \end{equation}
where the second map is the residue map from \cite{BK86}. 

Such results have appeared in the literature before,
but usually only in low weights or with some denominators, using the approach to
syntomic cohomology
of
\cite{FM87, Kato}, cf.~\cite[Sec.~6]{AMNN} for a comparison. 
In particular, \cite{Kurihara} constructs the above cofiber sequence in low
weights. 
The comparison for semistable schemes and more generally with a log structure after allowing
denominators (in all weights) is \cite{CN17}. 
Integral comparisons 
for algebras over $\mathcal{O}_C$ appear in the smooth case in 
\cite[Th.~10.1]{BMS2} and in the semistable case (allowing log structures) in
\cite{CDN21}; up to isogeny or in low weights, 
this was previously treated in \cite{Kato, Tsuji}. 
\end{example}

\bibliographystyle{amsalpha}
\bibliography{motivic}
\end{document}